\documentclass[a4paper,reqno]{amsart}

\usepackage[a4paper,width={16cm},left=2.5cm,bottom=3cm, top=3cm]{geometry}
\usepackage{enumerate}
\usepackage{yfonts}

\usepackage[english]{babel}

\usepackage{a4wide}
\usepackage[dvipsnames]{xcolor}
\usepackage{amsmath}
\usepackage{amssymb}
\usepackage{amsfonts}
\usepackage{amsthm,graphicx}
\usepackage{comment}
\usepackage{csquotes}

\usepackage{hyperref}
\hypersetup{
	linktocpage=true,
	colorlinks=true,
	linkcolor=blue!70!black,
	citecolor=orange!40!red,
	urlcolor=magenta!80!black,
}

\usepackage{mathtools}
\mathtoolsset{showonlyrefs}

\numberwithin{equation}{section}
\newtheorem{theorem}{Theorem}[section]
\newtheorem{lemma}[theorem]{Lemma}

\newtheorem{proposition}[theorem]{Proposition}
\theoremstyle{definition}

\newtheorem{example}[theorem]{Example}
\newtheorem{remark}[theorem]{Remark}

\newcommand{\mc}{\mathcal}
\newcommand{\mb}{\mathbb}

\newcommand{\la}{\lambda}
\newcommand{\norm}[1]{\left\lVert#1\right\rVert}
\newcommand{\pd}[2]{\frac{\partial#1}{\partial#2}}

\newcommand{\R}{\mb{R}}

\newcommand{\e}{\varepsilon}

\newcommand{\D}{\Delta}

\newcommand{\ps}[3]{\left( #2, #3 \right)_{#1}}

\newcommand{\inn}{\textnormal{in}\ }
\newcommand{\onn}{\textnormal{on}\ }
\newcommand{\dvol}{\ \textnormal{dv}}
\newcommand{\da}{\ \textnormal{da}}
\newcommand{\lnorm}[3]{| | #1 | |_{L^{#2}(#3)}}
\newcommand{\hnorm}[3]{| | #1 | |_{H^{#2}(#3)}}

\newcommand{\capa}[2]{\textnormal{\rm{Cap}}_{#1}(#2)}

\usepackage{subfig}
\usepackage{tikz}
\usepackage{xcolor}
\usepackage{pgfplots}
\pgfplotsset{compat=1.18}
\usepgfplotslibrary{fillbetween}
\usepackage{mathrsfs}
\usetikzlibrary{arrows}
\usetikzlibrary{shadings}
\makeatletter

\begin{document}
\title[Quantitative spectral stability for compact operators]
{Quantitative spectral stability for compact operators}

\author[A. Bisterzo, G. Siclari]{Andrea Bisterzo  and Giovanni Siclari}

\address{Andrea Bisterzo
  \newline \indent Dipartimento di Scienze di Base e Applicate per l'Ingegneria
  \newline \indent
Sapienza Università di Roma
\newline\indent Via Antonio Scarpa 14, 00161, Roma, Italy.}
\email{andrea.bisterzo@uniroma1.it}

\address{ Giovanni Siclari 
  \newline \indent Dipartimento di Matematica
  \newline \indent Politecnico di Milano
  \newline\indent Piazza Leonardo da Vinci 32, 20133, Milano, Italy}
\email{giovanni.siclari@polimi.it}

\date{\today}

\begin{abstract}
This paper deals with quantitative spectral stability for compact operators acting on $L^2(X,m)$, where $(X,m)$ is a measure space. Under fairly general assumptions,  we provide a characterization of the dominant term of the asymptotic expansion of the eigenvalue variation in this abstract setting.  Many of the results about  quantitative spectral stability available in the literature can be recovered by our analysis.  Furthermore, we illustrate our result with several applications, e.g. quantitative  spectral stability for a Robin to Neumann problem, conformal transformations of Riemann metrics, Dirichlet forms under the removal of sets of small capacity, and  for families  of pseudo-differentials operators.
\end{abstract}

\maketitle

\hskip10pt{\footnotesize {\bf Keywords.} Compact operators, spectral theory, asymptotics of eigenvalues.

\medskip 

}

\section{Introduction}
Let us consider the Hilbert space $L^2(X,m)$, where $(X,m)$ is a measure space endowed with a positive measure $m$.
Given $\{(Y_\e, m_\e)\}_{\e\in [0,1]}$ a family of measure spaces, where $Y_\e \subseteq X$ and $m_\e$ are positive measures over $Y_\e$, let $\{H_\e\}_{\e \in [0,1]}$ be a family of   non-negative,   self-adjoint operators
\begin{equation}\label{def_He}
H_\e:D[H_\e]\subseteq L^2(Y_\e,m_\e)\to L^2(Y_\e,m_\e),
\end{equation}
with $m_0=m$ and $Y_0=X$.
Moreover, we  suppose that the spectrum $\sigma(H_\e)$ of $H_\e$ is discrete for any $\e \in[0,1]$ and we denote it by
$\sigma(H_\e)=\{\la_{n,\e}\}_{n \in \mb{N}}$, where we are repeating each eigenvalues according to its multiplicity.

The present paper is aimed to the study of the asymptotic of  the eigenvalue's variation 
$\la_{n,\e}-\la_{n,0}$ as $\e \to 0^+$ for  simple eigenvalues, under suitable assumptions on the family $\{H_\e\}_{\e \in [0,1]}$. 
In particular we are interested in a \textit{quantitative} result, that is, in \textit{how} the eigenvalues $\lambda_{n,\e}$ converge to $\lambda_{n,0}$ as $\e\to 0^+$.

There is a vast literature about quantitative spectral stability for compact operators, with many examples  in the field of elliptic PDEs on the Euclidean space or on Riemann manifolds, in particular under smooth and nonsmooth deformation of the domain, see for  example   the classical paper \cite{RT_qualitative_stability}. We also refer to  \cite{H_eigen} for an overview.

A very classic problem is quantitative spectral stability for elliptic operators under some geometrical perturbation on the domain,
for example, removing a small set and  imposing Dirichlet boundary condition on the removed set.

In the seminal paper \cite{O_Dirichlet_holes_euclidean}, the variation of the eigenvalues of the Laplace operator acting on an Euclidean domain $D$, and perturbed by removing small balls, is characterized explicitly by the means of Green functions. For generic holes concentrating in a compact set $K$, the eigenvalue's variation is characterized  in terms of weighted capacity of the removed sets  in \cite{FNO_disa_Dirichlet_region}.
In dimension $2$, a precise first order asymptotic expansion of the eigenvalue's variation for similar problems is obtained in \cite{AFHC_spec_AB} and in \cite{ALM_cap_exp}. In particular, in this latter work the authors provided an expansion in series of the weighted capacity of a scaling hole $\e K$.
In an arbitrary dimension $N$, for scaling holes $\e K$, it is possible to explicitly estimate the vanishing order, see \cite{FNO_disa_Dirichlet_region}. 
In the general setting of Riemann manifolds we refer to 
 \cite{C_mainfols_holes} for a precise description of the eigenvalue's variation in terms of the capacity of the removed sets and for a more detailed overview of the literature for the Dirichlet Laplacian on manifolds.

For the fractional Laplacian operator in a bounded domain of $\mathbb{R}^N$, a similar characterization of the eigenvalue's variation in terms of a weighted  fractional capacity was proved in \cite{AFN_fractional}. Furthermore, in the same paper, the author managed to compute the eigenvalue's variation in the case of a sequence of shrinking holes of the form $\e K$, where $K$ is a compact subset. Such computation is sharp in the case of set of positive $N$-dimension  Lebesgue measure.

For the ground state of a generic Dirichlet form $(\mc{E}, \mc{F})$ in a measure space $(X,m)$, under the removal of sets of small capacity, the eigenvalue's variation is expanded at the second order in \cite{G_Dirichlet_forms}, under suitable assumption on the resolvent and semigruop associated to $\mc{E}$.

For the polyarmonic operator $(-\Delta )^m$ acting on a bounded Euclidean domain, the eigenvalue's variation is computed in \cite{FR_polyharmonic} in terms of a suitable notion of weighted capacity. If the removed sets are given by  $\e K$ for some compact set $K$, a more precise description of the eigenvalue's variation, with explicit estimates on its vanishing 
order, is also provided in the aforementioned work.

The Neumann  Laplacian operator in a bounded domain $\Omega\subset \R^N$, where  Neumann conditions are imposed  on the boundary of the removed set,  was studied for example in  \cite{L_Neumann_analyt,FLO_holes_Neumann}. 
In \cite{L_Neumann_analyt} removing a  set of size $\e$, the author expresses the eigenvalue of the perturbed problem as an analytic function in two variables composed with the perturbing parameter $\e$ and $\delta_N \e \log(\e)$, where $\delta_N= 1$ if $N$ is odd and $0$ if $N$ is even.
In \cite{FLO_holes_Neumann}  the eigenvalue's variation is  characterizes by the   \textit{  weighted torsion } of the removed set. Furthermore, for holes of the form $x_0+\e K$ for some $x_0 \in \Omega $ and compact $K \subset \Omega$, the  eigenvalue's variation can be computed more explicitly. It is worth noticing that its vanishing order strongly depends on the point $x_0$. A completely explicit expression of the eigenvalue's variation is obtained for spherical holes.
We also mention that the study of the convergence in Hausdorff distance of the spectrum for generics holes in $\R^2$ was studied in \cite{BSH_neum_spect}.

A different case of geometric perturbation was consider in \cite{FO_tubes}, for the Euclidean Dirichlet Laplacian. More precisely the authors attached a small cylindrical tube to the unperturbed domain $\Omega$ and let the section of  the cylinder shrinks. The vanishing order of the eigenvalue's variation is computed with monotonicity formula's  techniques.

The case where mixed Dirichlet and Neumann boundary condition are imposed  on the boundary of a  bounded domain $\Omega\subset \mathbb{R}^N$ for the Laplacian operator were consider in \cite{FNO_disa_Dirichlet_region,FNO_disa_Neuman_region}. More precisely, if the  Dirichlet region vanishes, then the  eigenvalue's variation can be quantified in terms of a weighted capacity of a disappearing  set. With  techniques based on a combination min-max estimates and  
a monotonicity formula, the case of a vanishing Neumann region was studied in \cite{FNO_disa_Neuman_region}. 

Finally, we mention that, for Aharonov-Bohm operators with an arbitrary numbers of moving poles, the eigenvalue's variation was recently studied in \cite{FNOS_AB_12} and \cite{FNS_AB}. In this case, the quantity that  characterizes the eigenvalue's variation may be seen as an intermediate notion between a weighted torsion and a weighted capacity.
There is a vast literature about quantitative spectral stability for  Aharonov-Bohm operators, see  \cite{FNOS_AB_12} for a detailed overview.

\bigskip

In the present paper, see Theorem \ref{theo_exp_1_order}, we  characterize the eigenvalue's variation as 
\begin{equation}
\la_{n,\e}-\la_{n,0}=\frac{\la_0\int_{Y_\e} \phi_{n,0} V_{n,\e} \ \textnormal{d}m_\e +O\left(\norm{V_{n,\e}}^2_{L^2(Y_\e,m_\e)}\right)}{\int_{Y_\e} |\phi_{n,0}|^2 \, \textnormal{d}m_\e+O(\norm{V_{n,\e}}_{L^2(Y_\e,m_\e)})}
\quad \text{ as  } \e \to 0,
\end{equation}
where $\phi_{n,0}$ is a normalized eigenfunction of the operator $H_0$ associated to the simple eigenvalue $\la_{n,0}$ and $V_{\e,n}$ is the unique minimizer of an energy functional $J_{\e,n}$ which depends on $\phi_{n,0}$, see \eqref{def_J_ef},
and we also  obtain the convergence of the associated  eigenfunctions in a suitable sense.

Theorem \ref{theo_exp_1_order} allows us to recover many of the previously mention results. In particular we are referring to
\cite[Theorem 5.2]{FNOS_AB_12},  \cite[Theorem 5.3]{FNS_AB}, \cite[Theorem 2.3]{FLO_holes_Neumann} and, see Subsection \ref{subsect_Diri_forms} for details, to \cite[Theorem 2.5]{FNO_disa_Dirichlet_region}, \cite[Theorem 1.5]{AFN_fractional} and \cite[Thoerem 1.2]{FR_polyharmonic}. Furthermore  Theorem \ref{theo_exp_1_order} generalizes \cite[Theorem 3.6]{G_Dirichlet_forms} to the case of simple eigenvalues and removes several assumptions, see Subsection \ref{subsect_Diri_forms}.

We also presents several applications of Theorem \ref{theo_exp_1_order}.

The paper is organized as follows. In Section \ref{sec_assump_preli} we provide some abstract preliminaries results and present the assumptions needed to prove Theorem  \ref{theo_exp_1_order}. In Section \ref{sec_first_order_exp} we prove Theorem \ref{theo_exp_1_order} and discuss several results usefull to ensure the validity of its assumption, focusing  on qualitative spectral stability in Subsection \ref{subsec_spectral_stability}. Finally, in Section \ref{sec_applications}, as applications of  Theorem  \ref{theo_exp_1_order},   we  study quantitative  spectral stability for a Robin to Neumann problem, conformal transformations of Riemann metrics, Dirichlet forms under the removal of sets of small capacity, and for  families of pseudo-differentials operators.

\section{Preliminaries and assumptions}\label{sec_assump_preli} 
Throughout the paper, for any measure  space $(X,m)$ we will  denote with $(\cdot,\cdot)_{L^2(X,m)}$ the scalar product in $L^2(X,m)$. 
Furthermore we denote the resolvent and the spectrum of an operator $H$ respectively by $\rho(H)$ and $\sigma(H)$, while we will use $D[H]$ for its domain. We also recall that  a symmetric, positive form $(\mc{E},\mc{F})$ on $L^2(X,m)$ is \textit{closed} if its domain $\mc{F}$ is complete with respect to the norm induced by the scalar product
\begin{equation}\label{def_E1}
  \mc{E}_1:=\mc{E}+(\cdot,\cdot)_{L^2(X,m)}.   
\end{equation}
To any positive, self-adjoint and densely defined  operator   
\begin{equation}
H:D[H]\subseteq L^2(X,m)\to L^2(X,m),
\end{equation}
it is possible to associate a unique symmetric bilinear form  $(\mc{E},\mc{F})$ given by
\begin{equation}\label{def_E}
\mc{E}(u,v):=(\sqrt{H} u,\sqrt{H} v)_{L^2(X,m)} \quad \text{ for any } u,v \in \mc{F},
\end{equation}
where $\mc{F}:=D[\sqrt{H}]$. Moreover,  the form $(\mc{E},\mc{F})$ is closed. On the other hand, any positive densely defined, closed symmetric forms admits a unique  operator $H$ such that $\eqref{def_E}$ holds; see for example \cite[Theorem 1.3.1]{FM_Dirichlet_forms_book}. 

\subsection{Preliminaries}\label{subsec_preliminaries} In this subsection we present some abstract results that will be used throughout the paper. For the sake of completeness we also provide a short proof of each result.
\begin{proposition}\label{prop_compact_embedding}
Let $(\mc{E},\mc{F})$ be a symmetric, positive, densely defined and closed form on $L^2(X,m)$ and let $H$ be the associated densely defined, positive and self-adjoint operator. 
Suppose $\la:=\min{\sigma(H)}>0$ is an eigenvalue of $H$ and that 
\begin{equation}
H^{-1}: L^2(X,m) \to  L^2(X,m) \quad \text{ is a compact operator.}
\end{equation}
Then the embedding $i:\mc{F} \hookrightarrow L^2(X,m)$ is compact.
\end{proposition}
\begin{proof}
Since $H^{-1}$ is compact, by \cite[Theorem 12.30]{R_functional_analysis_book} and the Spectral Theorem for bounded, self-adjoint operators, $H^{-\frac{1}{2}}$ is a compact operator. Let $v_n \rightharpoonup v$ weakly in $\mc{F}$ as $n \to \infty$. Then for any $w \in \mc{F}$
\begin{equation*}
\lim_{n\to \infty}\ps{L^2(X,m)}{\sqrt{H}v_n}{\sqrt{H}w}=\lim_{n\to \infty}\mc{E}(v_n,w)=\mc{E}(v,w)
= \ps{L^2(X,m)}{\sqrt{H}v}{\sqrt{H}w}.
\end{equation*}
Hence  $\sqrt{H}v_n \rightharpoonup \sqrt{H} v$ weakly in $L^2(X,m)$  as $n \to \infty$
and so, by the compactness of $H^{-\frac{1}{2}}$, $v_n \to v$ strongly in $L^2(X,m)$ as $n \to \infty$.   
\end{proof}

In the following we will make use of the next (abstract) Poincare's inequality. 
\begin{proposition}\label{prop_poin}
Let $(\mc{E},\mc{F})$ be a symmetric, densely defined, positive and closed form on $L^2(X,m)$ and  let $H$ be the associated densely defined, positive and self-adjoint operator. Suppose  that $\la:=\min{\sigma(H)}>0$ is an eigenvalue of $H$ and that $H^{-1}$ is  a compact operator. Then  
\begin{equation}\label{ineq_poinc}
\norm{w}^2_{L^2(X,m)} \le \la^{-1} \mc{E}(w,w) \quad \text{ for any } w \in \mc{F}.
\end{equation}
\end{proposition}
\begin{proof}
Let us consider the minimization problem
\begin{equation}\label{min_Rel}
\inf \left\{\frac{\mc{E}(w,w)}{\norm{w}^2_{L^2(X,m)}}: w \in \mc{F}\setminus\{0\}\right\}.
\end{equation}
Let us show that there exists a minimizer  of  \eqref{min_Rel}.
Let $\{u_n\}_{n \in\mathbb{N}}$ be a minimizing sequence   and let 
\begin{equation}
v_n:=\frac{u_n}{\norm{u_n}_{L^2(X,m)}}.
\end{equation}
Then also  $\{v_n\}_{n \in\mathbb{N}}$ is a minimizing sequence and $\norm{v_n}_{L^2(X,m)}=1$. 
Since $\mathcal{E}$ is closed (hence $(\mathcal{E}_1,\mathcal{F})$ is Hilbert) and $\{v_n\}_n$ is bounded, up to a subsequence, we can assume that $v_n \rightharpoonup v$  weakly in the  norm $\mc{E}_1=\mc{E}+\ps{L^2(X,m)}{\cdot}{\cdot}$ as $n \to \infty$ to some $v \in \mc{F}$.

In particular  $\norm{v}_{L^2(X,m)}=1$ thanks to  Proposition \ref{prop_compact_embedding}.
By the lower semicontinuity of the norm $\mc{E}_1$ we have
\begin{align*}
 \mc{E}_1(v,v)\leq \liminf_{n\to \infty} \mc{E}_1(v_n,v_n)
\end{align*}
and so
\begin{equation*}
\mc{E}(v,v) \le\liminf_{n\to \infty} \mc{E}(v_n,v_n) .
\end{equation*}
We conclude that $v$ minimizes \eqref{min_Rel}. Hence, $v$ is a solution of the equation
\begin{equation*}
\mc{E}(v,w)=\frac{\mc{E}(v,v)}{\norm{v}^2_{L^2(X,m)}}\ps{L^2(X,m)}{v}{w} \quad \text{ for any } w \in \mc{F}.
\end{equation*}
It follows that $\frac{\mc{E}(v,v)}{\norm{v}^2_{L^2(X,m)}}$ is an eigenvalue of $H$ and so $\la \le \frac{\mc{E}(v,v)}{\norm{v}^2_{L^2(X,m)}}$.
Evaluating the ratio in \eqref{min_Rel} with an eigenfunction associated to $\la$ we obtain the opposite inequality thus concluding that
\begin{equation}
\inf \left\{\frac{\mc{E}(w,w)}{\norm{w}^2_{L^2(X,m)}}: w \in \mc{F}_\e\setminus\{0\}\right\}=\la,
\end{equation}
which proves \eqref{ineq_poinc}.
\end{proof}

We will use the following regularity result in  Subsection \ref{subsec_pseudo}.

\begin{proposition}\label{prop_regularity}
Let $(\mc{E}, \mc{F})$ be a positive, densely defined, symmetric and closed bilinear form and $H$ the correspondent densely defined, positive and self-adjoint operator. Let $f \in L^2(X,m)$ and suppose that $u \in \mc{F}$ solves the equation
\begin{equation}\label{eq_regularity}
\mc{E}(u,v)=\ps{L^2(X,m)}{f}{v} \quad \text{ for any } v \in \mc{F}.
\end{equation}
Then $u \in D[H]$ and $Hu=f$.
\end{proposition}

\begin{proof}
If $f=0$ then $u=0$. Hence, up to a renormalization, it is not restrictive to suppose that $\norm{f}_{L^2(X,m)}=1$.
Since $H$ is positive and self adjoint, by the Spectral Theorem, see \cite[Theorem 13.30]{R_functional_analysis_book}, the operators $T_t:=e^{-tH}$ are linear, continuous, and self- adjoint on $L^2(X,m)$ for any $t>0$. Furthermore, $T_t$ commutes with the powers $H^\alpha$ of $H$ for any positive $\alpha\in \mathbb{R}$ and also $H^\alpha T_t$ is linear, and continuous on $L^2(X,m)$ for any $t>0$. Finally $\{T_t\}_{t>0}$ forms a strongly continuous semigroup, see for example \cite[Lemma 1.3.2]{FM_Dirichlet_forms_book}.
In particular, for any $t>0$ we may test \eqref{eq_regularity} with $H T_tu$ thus obtaining 
\begin{equation}
\ps{L^2(X,m)}{f}{H T_tu}=\mc{E}(u,H T_tu)=  (\sqrt{H} u,\sqrt{H} H T_tu)_{L^2(X,m)}=\norm{H T_{\frac{t}{2}}u}_{L^2(X,m)}^2,
\end{equation}
in view of  \cite[Theorem 1.3.1]{FM_Dirichlet_forms_book}.
It follows that 
\begin{equation}
\norm{H T_{\frac{t}{2}}u}_{L^2(X,m)}  \le 1.
\end{equation}
In particular there exists a sequence $t_n\to 0^+$ and $w \in L^2(X,m)$ such that $H T_{\frac{t_n}{2}}u \rightharpoonup w$ weakly in $L^2(X,m)$ as $n \to \infty$.
Hence 
\begin{equation}\label{proof:prop_regularity_2}
\lim_{n \to \infty}\norm{H T_{\frac{t_n}{2}}u}_{L^2(X,m)}^2 =\ps{L^2(X,m)}{f}{w} \le \norm{w}_{L^2(X,m)}
\end{equation}
while, by the lower semicontinuity of the $L^2$- norm,
\begin{equation}\label{proof:prop_regularity_1}
\norm{w}_{L^2(X,m)} \le \lim_{n \to \infty}\norm{H T_{\frac{t_n}{2}}u}_{L^2(X,m)}.
\end{equation}
Furthermore, by \cite[Theorem 1.3.1]{FM_Dirichlet_forms_book}, for any $v \in \mc{F}$ we have that
\begin{equation}
\ps{L^2(X,m)}{w}{v}=\lim_{n \to \infty}(H T_{\frac{t_n}{2}}u, v)_{L^2(X,m)}=\lim_{n \to \infty}(T_{\frac{t_n}{2}}\sqrt{H}u,\sqrt{H} v)_{L^2(x,m)}= \mc{E}(u,v).
\end{equation}
Hence by \eqref{eq_regularity} we have that $w=f$ and in particular $\norm{w}_{L^2(X,m)}=1$. Then, thanks to \eqref{proof:prop_regularity_2} and \eqref{proof:prop_regularity_1}, we conclude that $ H T_{\frac{t_n}{2}} u \to w$ strongly in $L^2(X,m)$ as $n \to \infty$.
Since $H$ is self-adjoint, its graphs is closed (see for example  \cite[Theorem 13.9]{R_functional_analysis_book}). Hence $u \in D[H]$ and $f=w=Hu$.
\end{proof}

\subsection{Assumptions}\label{subsec_assumptions}
For any $\e \in [0,1]$ let $Z_\e$ be a linear subspace of $\mc{F}_\e$ that is closed with respect to the norm  $(\mc{E}_1^{(\e)})^{\frac{1}{2}}$, where $\mc{E}_1^{(\e)}$ is as in \eqref{def_E1}. Then let us  define 
\begin{equation}\label{def_Ee}
\overline{\mc{E}^{(\e)}}:=\mc{E}^{(\e)}_{|Z_\e \times Z_\e}.
\end{equation}
It is clear that $\overline{\mc{E}^{(\e)}}$ is a closed, positive definite and symmetric form with domain $Z_\e$.
If we define
\begin{equation}\label{def_mc_Ze}
\mc{Z}_\e:=\overline{Z_\e}^{\norm{\cdot}_{L^2(Y_\e,m_\e)}} 
\end{equation}
that is, $\mc{Z}_\e$ is  the closure of $Z_\e$ with respect to the norm of $L^2(Y_\e,m_\e)$ and we endow it with the norm induced by $L^2 (Y_\e,m)$, then clearly $\overline{\mc{E}^{(\e)}}$ is densely defined in $\mc{Z}_\e$.
Let $\overline{H_\e}$  be the associated  positive, densely definite and self adjoint operator, see \cite[Theorem 1.3.1]{FM_Dirichlet_forms_book}.
In what follows we suppose that for any $\e \in [0,1]$
\begin{equation}\label{A2}\tag{A1}
\begin{gathered}
\textnormal{the spectrum of  $\overline{H_\e}$ is discrete and it consists of a diverging}\\
\textnormal{sequence of non-negative eigenvalues $\{\la_{\e,n}\}_{n \in \mb{N}\setminus\{0\}}$,}
\end{gathered}
\end{equation}
where each eigenvalue is repeated according to its multiplicity.

Under the assumption \eqref{A2}, let $\{\phi_{\e,n}\}_{n \in \mb{N}\setminus\{0\}}$ be the associated basis of $\mc{Z}_\e$ given by the eigenfunctions of $\overline{H_\e}$. In  particular,
\begin{equation}\label{eq_phie_eigen}
\overline{\mc{E}^{(\e)}}(\phi_{n,\e}, w)= \la_{n,\e} \ps{L^2(Y_\e,m_\e)}{\phi_{n,\e}}{w} \quad \text{ for any } w \in Z_\e.
\end{equation}
Since we are interested in the quantitative spectral stability, it is natural to suppose that  
\begin{equation}\label{hp_limit_simple}\tag{A2}
\lim_{\e \to 0} \la_{\e,n}=\la_{0,n}, \quad \text{ for any } n \in \mathbb{N}\setminus \{0\}
\end{equation}
In addition, it is not restrictive to suppose that 
\begin{equation}\label{hp_la_1>0}\tag{A3}
\la_{\e,1}>0 \quad \text{ for any } \e \in[0,1].
\end{equation}
Indeed if $\la_{\e,1}=0$ then we can perform our analysis for the family of operators $\{\overline{H_\e}+\mathop{\rm{Id}}\}_{\e \in [0,1]}$.
Clearly $\{\overline{H_\e}+\mathop{\rm{Id}}\}_{\e \in [0,1]}$ is a family of densely defined, positive, self-adjoint operators that possesses the  same  properties of 
$\{\overline{H_\e}\}_{\e \in [0,1]}$. The spectrum of ${\overline{H_\e}+\mathop{\rm{Id}}}$ is given by $\{\la_{\e,n}+1\}_{n \in \mathbb{N}\setminus\{0\}}$.

It is a standard fact that, under the above assumptions, the operator $(\overline{H_\e})^{-1}$ is continuous, positive and self-adjoint. Moreover, in view of \cite[Theorem 12.30]{R_functional_analysis_book}, by \eqref{hp_la_1>0} the operator $(\overline{H_\e})^{-1}$ also turns out to be compact.

Fix a simple eigenvalue $\la_{0,n}$ and a correspondent eigenfunction $\phi_{0,n}$ For the sake of simplicity we write $\la_0$ instead of $\la_{0,n}$ and  $\phi_0$ instead of $\phi_{0,n}$.
We assume that
\begin{align}\label{hp_m_Ye}\tag{A4}
\norm{\phi_0}_{L^2(X,m)}=1 \quad \text{ and } \quad \lim_{\e \to 0}\int_{Y_\e} |\phi_0|^2 \ \textnormal{d}m_\e=1.
\end{align}
We suppose that the space $\mc{F}_0$ is included in any $\mc{F}_\e$ in the following sense
\begin{equation}\label{hp_uin_F0_u_in_Fe} \tag{A5}
   u \in \mathcal{F}_0\quad \Rightarrow \quad u_{|Y_\e} \in \mathcal{F}_\e.
\end{equation}
Furthermore, we assume that  for any $\e \in (0,1]$ there exists a linear functional $L_\e \in (\mc{F}_\e)^*$ 
such that 
\begin{equation}\label{eq_Vehi0_in_Ze}\tag{A6}
\mc{E}^{(\e)}(\phi_0,u)=\la_0 (\phi_0,u)_{L^2(Y_\e,m_\e)} +L_\e(u) \quad \text{ for any } u \in Z_\e,
\end{equation}
which is a well-posed assumption in view of \eqref{hp_uin_F0_u_in_Fe}.
For the sake of simplicity, we will still denote with $\mc{E}^{(\e)}$ the quadratic form 
\begin{equation}\label{def_Ee_quadratic}
\mc{E}^{(\e)}(w):=  \mc{E}^{(\e)}(w,w) \quad \text{ for any } \e \in [0,1] \text{ and } w \in \mc{F_\e}.
\end{equation}
Let us define for any $\e \in (0,1]$ the functional 
\begin{equation}\label{def_J_ef}
J_{\e}(u):= \frac{1}{2}\mc{E}^{(\e)}(u)- L_{\e}(u) \quad \text{ for any } u \in \mc{F}_\e.
\end{equation}

\begin{proposition}\label{prop_J_min}
For any $\e \in [0,1]$ there exists a unique $V_\e \in  \mc{F}_\e$ solving the minimization problem 
\begin{equation}\label{prob_J_min}
\inf\{J_\e(u):u \in Z_\e+{\phi_0}_{|Y_\e}\}.
\end{equation}
Furthermore, $V_\e$ is the unique solution to the equation 
\begin{equation}\label{eq_Ve}
\mc{E}^{(\e)}(V_\e,u)=L_\e(u)\quad  \text{ for any } u \in Z_\e
\end{equation}
such that $V_\e-{\phi_0}_{|Y_\e} \in Z_\e$. If 
\begin{equation}\label{hp_Ve_not_0}
L_\e \not\equiv 0 \text{ in } Z_\e \quad \text{ or } \quad {\phi_0}_{|Y_\e} \not \in Z_\e.
\end{equation}
then $V_\e \neq 0$.
\end{proposition}

\begin{proof}
Thanks to \eqref{ineq_poinc}, denoting by $||\cdot||_{Z_\e}$ the operator norm over $Z_\e$, it is enough to notice that 
\begin{equation}
J_\e(u) \ge \frac{1}{2}\mc{E}^{(\e)}(u)-(\mc{E}_1^{(\e)}(u))^{\frac{1}{2}}\norm{L_\e}_{(\mc{F}_\e)^*}\ge \frac{1}{2}\mc{E}^{(\e)}(u)-(1+\la_{\e,1})(\mc{E}^{(\e)}(u))^{\frac{1}{2}}\norm{L_\e}_{(\mc{F}_\e)^*}
\end{equation}
to see that $J_\e$ is coercive on $\mc{F}_\e$. Since $J_\e$ is also convex and  continuous on $\mc{F}_\e$ we obtain the existence of a minimizer $V_\e \in \mc{F}_\e$ of \eqref{prob_J_min} which solves \eqref{eq_Ve}. 

Furthermore, if $v_1$ and $v_2$ solve \eqref{eq_Ve} and $v_i-{\phi_0}_{|Y_\e} \in Z_\e$ for $i=1,2$, then $v:=v_1-v_2$ solves 
\begin{equation}
\mc{E}^{(\e)}(v,u)=0  \quad \text{ for any } u \in Z_\e.
\end{equation}
Since $v \in Z_\e$, testing with $v$ we conclude that $v=0$ by \eqref{ineq_poinc}.

Finally, if  ${\phi_0}_{|Y_\e} \not \in Z_\e$ then  clearly $V_\e$ is not trivial. On the other hand, suppose that ${\phi_0}_{|Y_\e} \in Z_\e$. Then we are minimizing the functional $J_\e$ on the linear space $Z_\e$. Since $L_\e\not\equiv  0$, there exists a function $w \in Z_\e$ such that $L_\e(w)>0$. Then, choosing $t>0$ small enough, $J_\e(tw)=t^2\mc{E}^{(\e)}(w)- tL_{\e}(w)<0$.
We conclude that $J_\e(V_\e) <0$ and so $V_\e \neq 0$.
\end{proof}

The last assumption of the present section is that 
\begin{equation}\label{hp_limit_E_e_Ve_o}\tag{A7}
\lim_{\e \to 0}\norm{V_\e}_{L^2(Y_\e,m_\e)}=0. 
\end{equation}
In view of \eqref{eq_Vehi0_in_Ze} and \eqref{eq_Ve}, this is a natural requirement since we are interested in stability results for the spectrum of $H_0$.

\section{Quantitative spectral stability}\label{sec_first_order_exp}
In this section we prove our main theorem and some addition results to quantify the asymptotic of the eigenvalues variation more explicitly. Furthermore, we discuss the validity of its assumption presenting a  criterion  for qualitative spectral stability.
\subsection{Asymptotic of the eigenvalues variation}\label{subsection_main_result}
Let $\la_0$ and $\phi_0$ be as in Section \ref{sec_assump_preli} and for the sake of simplicity let $\la_\e:=\la_{n,\e}.$
In view of \eqref{hp_limit_simple}, it is not restrictive to suppose that $\la_\e$ is simple for any $\e \in [0,1]$. Let $\phi_\e:=\phi_{n,\e}$ be a eigenfunction corespondent to the eigenvalue $\la_{\e}$ such that
\begin{equation}\label{hp_phie_normalized}
\norm{\phi_\e}^2_{L^2(Y_\e,m_\e)}=1.
\end{equation}
Let  us define the orthogonal projection 
\begin{equation}\label{def_Pie}
\Pi_\e:L^2(Y_\e,m_\e) \to Z_\e, \quad \Pi_\e(w)=\ps{L^2(Y_\e,m_\e)}{w}{\phi_\e} \phi_\e.
\end{equation}
Since $(\overline{\mc{E}^{(\e)}}, Z_\e)$ is closed, in view of  Proposition \ref{prop_poin} it follows that $Z_\e$ is a Hilbert space respect to the scalar product $\overline{\mc{E}^{(\e)}}$.
Let 
\begin{equation}\label{def_Re}
{R_\e}:=(\overline{{H}_\e})^{-1}: \mc{Z}_\e \to Z_\e.
\end{equation}
By the very definition of $\overline{\mc{E}^{(\e)}}$, it holds
\begin{equation}\label{eq_Re}
\overline{\mc{E}^{(\e)}}(R_\e u,v )=(u,v)_{L^2(Y_\e,m_\e)} \quad \text{ for any } u,v \in Z_\e
\end{equation}
and so if we endow $Z_\e$ with  the norm $\overline{\mc{E}^{(\e)}}^\frac{1}{2}$ then $R_\e$  is  linear and continuous. In particular, the restriction of $R_\e$ to $Z_\e$, which we still denote with $R_\e$, is linear and continuous as well.
By the Spectral Theorem for bounded, self-adjoint operators it follows that
\begin{equation}\label{ineq_dist_spec}
(\mathop{{\rm{dist}}}(\mu,\sigma(R_\e)) )^2
\le \frac{\overline{\mc{E}^{(\e)}}(R_\e w -\mu w)}{\overline{\mc{E}^{(\e)}}(w)}
\quad  \text{ for any } \mu \in \rho(R_\e) \text{ and } w \in Z_\e\setminus\{0\},
\end{equation}
see for example \cite[Proposition 8.20]{H_book_spectral}.

The proof of the following theorem is inspired by \cite[Appendix A]{ACF_spectral_stability_AB}, \cite[Theorem 2.5]{FNO_disa_Dirichlet_region},
\cite[Theorem 1.5]{AFN_fractional} and \cite[Theorem 5.2]{FNOS_AB_12}.
\begin{theorem}\label{theo_exp_1_order}
Let $\{H_\e\}_{\e\in [0,1]}$ be a family of positive, densely defined and self-adjoint operators
\begin{align}
    H_\e :D[H_\e]\subseteq L^1(Y_\e,m_\e)\to L^2(Y_\e,m_\e)
\end{align}
with associated bilinear forms $(\mc{E}^{(\e)},\mathcal{F}_\e)$. Fixed a simple eigenvalue $\lambda_0:=\lambda_{0,n}$ of $H_0$ with associated normalized eigenfunction $\phi$, let $\la_\e:=\la_{\e,n}$.
Suppose that all the assumptions \eqref{A2} to \eqref{hp_limit_E_e_Ve_o} hold.

Then,
\begin{equation}\label{eq_asymptotic_eigenvlaues_1_order}
\la_\e-\la_0=\frac{\la_0\int_{Y_\e} \phi_0 V_\e \ \textnormal{d}m_\e +O\left(\norm{V_\e}^2_{L^2(Y_\e,m_\e)}\right)}{\int_{Y_\e} |\phi_0|^2 \, \textnormal{d}m_\e+O(\norm{V_\e}_{L^2(Y_\e,m_\e)})}
\quad \text{ as  } \e \to 0,
\end{equation}
where $V_\e$ is the function provided by Proposition \ref{prop_J_min}. Furthermore,
\begin{align}
\mc{E}^{(\e)}(\phi_0-V_\e-\Pi_\e(\phi_0-V_\e))=O\left(\norm{V_\e}^2_{L^2(Y_\e,m_\e)}\right)
\quad \text{ as }\e \to 0,\label{eq_eigenfunctions_Ve}
\end{align}
and
\begin{align}
&\norm{\phi_0-\Pi_\e(\phi_0-V_\e)}_{L^2(Y_\e,m_\e)}=O\left(\norm{V_\e}_{L^2(Y_\e,m_\e)}\right), \label{eq_eigenfunctions}\\ \ \\
&\mc{E}^{(\e)}(\phi_0-\Pi_\e(\phi_0-V_\e))=\mc{E}^{(\e)}(V_\e)+O\left(\norm{V_\e}_{L^2(Y_\e,m_\e)}(\mc{E}^{(\e)}(V_\e))^{\frac{1}{2}}\right)=O\left(\mc{E}^{(\e)}(V_\e)\right), \label{eq_eigenfunctions_E}
\end{align}
as $\e\to 0$.
\end{theorem}

\begin{proof}
Suppose that $V_\e=0$. Then $L_\e=0$ and ${\phi_0}_{|Y_\e} \in Z_\e$. It follows that ${\phi_0}_{|Y_\e}$ is an eigenfunction of $\mc{E}_\e$. Hence by \eqref{hp_limit_simple}, for $\e$ small enough, $\la_\e=\la_0$. In conclusion \eqref{eq_asymptotic_eigenvlaues_1_order} holds trivially.
Therefore it is not restrictive to suppose that $V_\e \neq 0$ for any $ \e \in (0,1]$.
 
Let $\psi_\e:=\phi_0-V_\e$.  
Then $\psi_\e \in Z_\e$ by Proposition \ref{prop_J_min} and 
\begin{equation}\label{eq_psie}
\overline{\mc{E}^{(\e)}}(\psi_\e, w)=\la_0\ps{L^2(Y_\e,m_\e)}{\phi_0}{w}  \quad\text{for all } w  \in Z_\e,
\end{equation}
in view of  \eqref{eq_Vehi0_in_Ze} and \eqref{eq_Ve}.
From \eqref{eq_psie} we deduce that 
\begin{equation}\label{proof:_heo_exp_1_order:1}
\overline{\mc{E}^{(\e)}}(\psi_\e, w)-\la_0\ps{L^2(Y_\e,m_\e)}{\psi_\e}{w}= \la_0\ps{L^2(Y_\e,m_\e)}{V_\e}{w} 
\quad\text{for all } w\in Z_\e.
\end{equation}
Since $\phi_{\e} \in Z_\e$, we may choose 
\begin{align}
    w:=\Pi_\e \psi_\e=(\psi_\e, \phi_\e)_{L^2(Y_\e,m_\e)}\phi_\e
\end{align}
in \eqref{proof:_heo_exp_1_order:1}, thus obtaining  \begin{equation}\label{proof_heo_exp_1_order:2}
(\la_\e-\la_0)\ps{L^2(Y_\e,m_\e)}{\psi_\e}{\Pi_\e \psi_\e}=
\la_0\ps{L^2(Y_\e,m_\e)}{V_\e}{\phi_0}+\la_0\ps{L^2(Y_\e,m_\e)}{V_\e}{\Pi_\e \psi_\e-\phi_0}. 
\end{equation}
by \eqref{eq_phie_eigen}.

Now we study the asymptotics, as $\e \to 0$, of each term in
\eqref{proof_heo_exp_1_order:2}. For the sake of simplicity, we divide
the rest of the proof into several steps.

\smallskip\noindent\textbf{Step 1.} We claim that  
\begin{equation}\label{eq_step1}
|\la_\e-\la_0|=O(\norm{V_\e}_{L^2(Y_\e,m_\e)}), \quad \text{as  }\e \to 0.
\end{equation}
Letting $\mu_0:=\la_0^{-1}$ and $\mu_\e:=\lambda_\e^{-1}$, since $\la_0$ is simple
and $\la_\e \to \la_0$ by \eqref{hp_limit_simple}, for $\e$ small enough
\begin{equation}\label{proof:_heo_exp_1_order:2}
\begin{split}|\la_\e-\la_0| & =\lambda_\e \lambda_0 |\mu_\e-\mu_0|\\
&\le 2
\lambda_0^2
\mathop{\rm{dist}}(\mu_0,\sigma(R_\e))\\
&\le 2 \lambda_0^2 \, \left(\frac{\overline{\mc{E}^{(\e)}}(R_\e \psi_\e-\mu_0 \psi_\e)}{\overline{\mc{E}^{(\e)}}(\psi_\e)}\right)^{\!1/2},
\end{split}
\end{equation}
thanks to \eqref{ineq_dist_spec}.
Since  $\norm{\phi_0}_{L^2(X,m)}=1$, by \eqref{eq_psie} and the Cauchy-Schwarz inequality we have 
\begin{equation}\label{proof:_heo_exp_1_order:3}
\begin{split}
\overline{\mc{E}^{(\e)}}(\psi_\e)&=\la_0\ps{L^2(Y_\e,m_\e)}{\phi_0}{\psi_\e} \\
&=\la_0-\lambda_0\left(1- \int_{Y_\e} |\phi_0|^2 \, \textnormal{d}m_\e\right) -
\la_0\int_{Y_\e} \phi_0 V_\e \, \textnormal{d}m_\e\\
&= \la_0 +o(1),
\end{split}
\end{equation}
where in the last equality we have used \eqref{hp_m_Ye}. By  \eqref{eq_Re} and \eqref{eq_psie}
tested with $R_\e \psi_\e-\mu_0 \psi_\e$, 
\begin{align*}
\overline{\mc{E}^{(\e)}}&(R_\e \psi_\e -\mu_0  \psi_\e)\\
&=-\ps{L^2(Y_\e,m_\e)}{V_\e}{R_\e \psi_\e-\mu_0 \psi_\e}  +\ps{L^2(Y_\e,m_\e)}{\phi_0}{R_\e \psi_\e-\mu_0 \psi_\e}-\overline{\mc{E}^{(\e)}}(\mu_0 \psi_\e,R_\e \psi_\e-\mu_0 \psi_\e)\\
&=-\ps{L^2(Y_\e,m_\e)}{V_\e}{R_\e \psi_\e-\mu_0 \psi_\e}.
\end{align*}
Hence, by the Cauchy-Schwarz inequality and Proposition \ref{prop_poin},
\begin{align*}
    \overline{\mc{E}^{(\e)}}(R_\e \psi_\e -\mu_0  \psi_\e) & =-\ps{L^2(Y_\e,m_\e)}{V_\e}{R_\e \psi_\e-\mu_0 \psi_\e}\\
    &\leq ||V_\e||_{L^2(Y_\e,m_\e)}\ ||R_\e \psi_\e-\mu_0 \psi_\e||_{L^2(Y_\e,m_\e)}\\
    &\leq ||V_\e||_{L^2(Y_\e,m_\e)}\ (\mu\  \overline{\mc{E}^{(\e)}}(R_\e \psi_\e-\mu_0 \psi_\e))^{\frac{1}{2}}.
\end{align*}
Hence
\begin{equation}\label{proof:_heo_exp_1_order:4}
\big(\overline{\mc{E}^{(\e)}}(R_\e \psi_\e-\mu_0  \psi_\e)\big)^{1/2}=O(\norm{V_\e}_{L^2(Y_\e,m_\e)})\quad\text{as }\e\to0.
\end{equation}
Claim \eqref{eq_step1} follows from \eqref{proof:_heo_exp_1_order:2}, \eqref{proof:_heo_exp_1_order:3}, and \eqref{proof:_heo_exp_1_order:4}.

\smallskip\noindent\textbf{Step 2.} We claim that  
\begin{equation}\label{eq_step2}
\overline{\mc{E}^{(\e)}}(\psi_\e-\Pi_\e\psi_\e)=O\left(\norm{V_\e}^2_{L^2(Y_\e,m_\e)}\right)  \quad \text{as  }\e \to 0.
\end{equation}
Let 
\begin{equation}\label{proof:_heo_exp_1_order:5}
\chi_\e:=\psi_\e-\Pi_\e\psi_\e \quad \text{and}  \quad \xi_\e:=R_\e \chi_\e-\mu_\e\chi_\e.
\end{equation}
By definition we have 
\begin{equation*}
\chi_\e \in N_\e:=\{w \in Z_\e:\ps{L^2(Y_\e,m_\e)}{w}{\phi_\e}=0\}
\end{equation*}
and from \eqref{eq_phie_eigen} and   \eqref{eq_Re} it follows that $R_\e w \in N_\e$ for all $w \in  N_\e$. Hence the operator 
\begin{equation*}
\widetilde{R}_\e:={R_\e}_{|N_\e}:N_\e\to N_\e
\end{equation*} 
is well-defined. Furthermore, $\sigma(\widetilde{R}_\e)=\sigma(R_\e)\setminus\{\mu_\e\}$. In particular, there
exists a constant $K>0$, which does not depends on $\e$, such that
$\big(\mathop{\rm{dist}}(\mu_\e,\sigma(\widetilde{R}_\e))\big)^2 \ge K$ (for $\e$ small enough). Moreover, \eqref{ineq_dist_spec} holds for ${\widetilde{R}}_\e$. Then
\begin{equation}\label{eq:estpsi-pi}
\begin{split}
\overline{\mc{E}^{(\e)}}(\chi_\e)&\le \frac{1}{K}\big(\mathop{\rm{dist}} (\mu_\e,\sigma(\widetilde{R}_\e))\big)^2\ 
\overline{\mc{E}^{(\e)}}(\chi_\e)  \\
&\le \frac{1}{K}\,\overline{\mc{E}^{(\e)}}(\widetilde{R}_\e \chi_\e-\mu_\e\chi_\e)\\
&=\frac{1}{K}\,\overline{\mc{E}^{(\e)}}(\xi_\e),    
\end{split}
\end{equation}
by \eqref{proof:_heo_exp_1_order:5}. In view of  \eqref{eq_psie} and \eqref{eq_phie_eigen} tested with $\xi_\e$ 
\begin{equation}\label{proof_heo_exp_1_order:6}
\overline{\mc{E}^{(\e)}}(\chi_\e,\xi_\e)-\la_\e\ps{L^2(Y_\e,m_\e)}{\chi_\e}{\xi_\e}=\la_{0}\ps{L^2(Y_\e,m_\e)}{V_\e}{\xi_\e}
+(\la_0-\la_\e)\ps{L^2(Y_\e,m_\e)}{\psi_\e}{\xi_\e}.
\end{equation}
Then from \eqref{eq_Re} and \eqref{proof_heo_exp_1_order:6} we
deduce that
\begin{align*}
  \overline{\mc{E}^{(\e)}}(\xi_\e)&=\overline{\mc{E}^{(\e)}}(R_\e \chi_\e,\xi_\e)-\mu_\e \overline{\mc{E}^{(\e)}}(\chi_\e,\xi_\e)\\
  &=-\mu_\e[ \overline{\mc{E}^{(\e)}}(\chi_\e,\xi_\e)-\la_\e \overline{\mc{E}^{(\e)}}(R_\e \chi_\e,\xi_\e)]\\
  &=-\frac{\la_{0}}{\la_\e}\ps{L^2(Y_\e,m_\e)}{V_\e}{\xi_\e}-\frac{(\la_0-\la_\e)}{\la_\e}
          \ps{L^2(Y_\e,m_\e)}{\psi_\e}{\xi_\e}.
\end{align*}
From the Cauchy-Schwarz inequality and Proposition \ref{prop_poin} it follows that
\begin{align*}
    \overline{\mc{E}^{(\e)}}(\xi_\e)\leq \frac{\lambda_0}{\lambda_\e^2} ||V_\e||_{L^2(Y_\e,m_\e)}(\overline{\mc{E}^{(\e)}}(\xi_\e))^{\frac{1}{2}}+\frac{|\lambda_\e-\lambda_0|}{\lambda_\e^2}||\psi_\e||_{L^2(Y_\e,m_\e)}(\overline{\mc{E}^{(\e)}}(\xi_\e))^{\frac{1}{2}}
\end{align*}
and hence, by \eqref{hp_limit_simple},
\begin{equation}\label{proof_heo_exp_1_order:7}
(\overline{\mc{E}^{(\e)}}(\xi_\e))^{\frac{1}{2}} 
\le C\left(\norm{V_\e}_{L^2(Y_\e,m_\e)}+|\la_\e-\la_0|\norm{\psi_\e}_{L^2(Y_\e,m_\e)}\right)
\end{equation}
for some constant $C>0$ which does not depend on $\e$. Furthermore,  \eqref{hp_limit_E_e_Ve_o} and \eqref{hp_m_Ye} yield 
\begin{equation*}
  \norm{\psi_\e}_{L^2(Y_\e,m_\e)}^2-1=o(1) \quad  \text{as } \e \to 0.
\end{equation*}
Then \eqref{eq_step2} follows from  \eqref{eq_step1}, \eqref{eq:estpsi-pi}, and \eqref{proof_heo_exp_1_order:7}. Since $\psi_\e= \phi_0-V_\e$ we have proved \eqref{eq_eigenfunctions_Ve}.

\smallskip\noindent\textbf{Step 3.} We claim that 
\begin{equation}\label{eq_step3}
\norm{\phi_0-\Pi_\e \psi_\e}_{L^2(Y_\e,m_\e)}=O(\norm{V_\e}_{L^2(Y_\e,m_\e)})\quad \text{ as  }\e \to 0.
\end{equation}
Indeed  from the definition of $\psi_\e$, Proposition \ref{prop_poin} and \eqref{eq_step2} it follows that
\begin{equation}\label{proof_heo_exp_1_order:9}
\begin{split}
\norm{\phi_0-\Pi_\e\psi_\e}_{L^2(Y_\e,m_\e)}& \le  \norm{\phi_0-\psi_\e}_{L^2(Y_\e,m_\e)}+  \norm{\psi_\e-\Pi_\e\psi_\e}_{L^2(Y_\e,m_\e)}\\
&=O(\norm{V_\e}_{L^2(Y_\e,m_\e)}) \quad \text{as  }\e \to 0.
\end{split}
\end{equation}
In particular we have proved \eqref{eq_eigenfunctions}. Furthermore, since $\phi_0=V_\e+\psi_\e $,
\begin{multline}
\mc{E}^{(\e)}(\phi_0-\Pi_\e\psi_\e)= \mc{E}^{(\e)}(V_\e)+ \mc{E}^{(\e)}(\psi_\e-\Pi_\e\psi_\e)+ 2\mc{E}^{(\e)}(V_\e,\psi_\e-\Pi_\e\psi_\e)\\
=\mc{E}^{(\e)}(V_\e)+O(\norm{V_\e}_{L^2(Y_\e,m_\e)}(\mc{E}^{(\e)}(V_\e))^{\frac{1}{2}})= O(\mc{E}^{(\e)}(V_\e)), \quad \text{as  }\e \to 0,  
\end{multline}
thanks to the Cauchy-Schwarz inequality, Proposition \ref{prop_poin} and \eqref{eq_step2}. Hence we have proved \eqref{eq_eigenfunctions_E}.

\smallskip\noindent\textbf{Step 4.} We claim that 
\begin{equation}\label{eq_step4}
\ps{L^2(Y_\e,m_\e)}{\psi_\e}{\Pi_\e\psi_\e}=\int_{Y_\e} |\phi_0|^2 \, \textnormal{d}m_\e +
O(\norm{V_\e}_{L^2(Y_\e,m_\e)}) \quad \text{as  }\e \to 0.
\end{equation}
We have that 
\begin{equation}
\begin{split}
\ps{L^2(Y_\e,m_\e)}{\psi_\e}{\Pi_\e\psi_\e}&=\ps{L^2(Y_\e,m_\e)}{\psi_\e-\Pi_\e\psi_\e}{\Pi_\e\psi_\e}+\norm{\Pi_\e\psi_\e}^2_{L^2(Y_\e,m_\e)}\\
&=\norm{\Pi_\e\psi_\e}^2_{L^2(Y_\e,m_\e)}.    
\end{split}
\end{equation}
Since $\norm{\phi_0}_{L^2(Y_\e,m_\e)}\xrightarrow[]{\e \to 0}1$, \eqref{proof_heo_exp_1_order:9} and the Cauchy-Schwarz inequality imply that
\begin{equation}\label{proof_heo_exp_1_order:10}
\begin{split}
\norm{\Pi_\e\psi_\e}^2_{L^2(Y_\e,m_\e)}&=
\norm{\phi_0-\Pi_\e\psi_\e}^2_{L^2(Y_\e,m_\e)}+  \norm{\phi_0}_{L^2(Y_\e,m_\e)}^2-2 \ps{L^2(Y_\e,m_\e)}{\phi_0-\Pi_\e\psi_\e}{\phi_0}\\
&=\int_{Y_\e} |\phi_0|^2 \, \textnormal{d}m_\e +O(\norm{V_\e}_{L^2(Y_\e,m_\e)}) \quad \text{as  }\e \to 0.    
\end{split}
\end{equation}
Putting together \eqref{proof_heo_exp_1_order:2}, \eqref{eq_step3} and \eqref{eq_step4}, we finally obtain
\begin{align*}
\la_\e-\la_0 &=\frac{\la_0\int_{Y_\e} \phi_0 V_\e \, \textnormal{d}m_\e +O\left(\norm{V_\e}^2_{L^2(Y_\e,m_\e)}\right)}{\int_{Y_\e} |\phi_0|^2 \, \textnormal{d}m_\e+O(\norm{V_\e}_{L^2(Y_\e,m_\e)})} \quad \text{as }\e \to 0,
\end{align*}
thus proving \eqref{eq_asymptotic_eigenvlaues_1_order}.
\end{proof}

\begin{remark}
It is also possible to prove \eqref{eq_asymptotic_eigenvlaues_1_order} and \eqref{eq_eigenfunctions} following the  approach exposed in \cite[Proof of Theorem 2.3]{FLO_holes_Neumann}. It is  based on a abstract lemma originally  proved  in \cite{C_lemma} and then revisited in \cite{C_mainfols_holes} and \cite{ACM_multiple_holes}. See \cite[Lemma 7.1]{FLO_holes_Neumann} for a simplified version of this lemma  suitable for simple eigenvalues and for a short proof.
However, following this approach we would only obtain 
\begin{equation}\label{eq_eigenfunctions_Ve_weak}
\norm{\phi_0-V_\e-\Pi_\e(v_0-V_\e)}_{L^2(Y_\e,m_\e)}=O\left(\norm{V_\e}_{L^2(Y_\e,m_\e)}\right),
\quad \text{ as }\e \to 0,
\end{equation}
which is a slightly weaker version of \eqref{eq_eigenfunctions_Ve}. To recover \eqref{eq_eigenfunctions_Ve}  some additional work is required, see \cite[Proof of Theorem 2.3]{FLO_holes_Neumann}.
\end{remark}

\begin{remark}\label{Rmk:Accurate1Exp}
It is worth noticing that, as $\e \to 0$, the quotient in \eqref{eq_asymptotic_eigenvlaues_1_order} can be written as
\begin{align*}
&\frac{\la_0\int_{Y_\e} \phi_0 V_\e \ \textnormal{d}m_\e +O\left(\norm{V_\e}^2_{L^2(Y_\e,m_\e)}\right)}{\int_{Y_\e} |\phi_0|^2 \, \textnormal{d}m_\e+O(\norm{V_\e}_{L^2(Y_\e,m_\e)})}\\
&=\left( \la_0\int_{Y_\e} \phi_0 V_\e \ \textnormal{d}m_\e +O\left(\norm{V_\e}^2_{L^2(Y_\e,m_\e)}\right) \right)\left(1+\frac{1-\int_{Y_\e} |\phi_0|^2 \ \textnormal{d}m_\e}{\int_{Y_\e} |\phi_0|^2 \, \textnormal{d}m_\e}+O(\norm{V_\e}_{L^2(Y_\e,m_\e)})\right)\\
&=\la_0\int_{Y_\e} \phi_0 V_\e \ \textnormal{d}m_\e +o\left(\norm{V_\e}_{L^2(Y_\e,m_\e)}\right),
\end{align*}
by the Cauchy-Schwarz inequality. Hence, if the rate of convergence of $||\phi_0||_{L^2(Y_\e,m_\e)}\xrightarrow{\e\to 0}1$ can be precisely quantified, a more accurate estimate on the asymptotic behavior of $\lambda_0-\lambda_\e$ can be recovered. In particular, if
\begin{align*}
1-\int_{Y_\e} |\phi_0|^2 \ \textnormal{d}m_\e=O\left(||V_\e||_{L^2(Y_\e,m_\e)}\right)\quad \textnormal{as}\ \e \to 0,
\end{align*}
then by the Cauchy-Schwarz inequality
\begin{align}
 \lambda_\e-\lambda_0 = \la_0\int_{Y_\e} \phi_0 V_\e \ \textnormal{d}m_\e +O\left(\norm{V_\e}^2_{L^2(Y_\e,m_\e)}\right)\quad \textnormal{as}\ \e \to 0.
\end{align}
\end{remark}

Under some additional assumptions, we are able to identify more explicitly a sequence of eigenfunctions   associated to the simple eigenvalues $\la_\e$, that converges in suitable sense to $\phi_0$. More precisely let  $\phi_\e$ be the eigenfunction associated to   $\la_\e$  such that 
\begin{equation}\label{hp_phi_e_diri}
\int_{Y_\e}|\phi_\e|^2 \, \textnormal{d}m_\e  =1 \quad \text{ and  } \quad \int_{Y_\e}\phi_0 \phi_\e \, \textnormal{d}m_\e>0. 
\end{equation}
Such a  choice is possible, at least for small $\e$, since in view of \eqref{eq_eigenfunctions}, $\int_{X} \phi_0 \phi_\e  \, dm \neq 0$ for any 
 $\e$ close enough to $0$.

\begin{proposition}\label{prop_conv_eigenfunctions}
Suppose that 
\begin{equation}
\int_{Y_\e} |\phi_0|^2 \ \textnormal{d}m_\e=1+ O\left(||V_\e||_{L^2(Y_\e,m_\e)}\right)
\quad \textnormal{ as } \e \to 0. \label{hp_conv_Y_e_eginefunctions}
\end{equation}
Then 
\begin{align}
&\mc{E}^{(\e)}(\phi_0-\phi_\e)= O\left(\mc{E}^{(\e)}(V_\e)\right)\quad \text{ as }  \e \to 0^+. \label{eq_conv_eigenfunctions_precisise}\\
\end{align}
Furthermore, if 
\begin{equation}
 \norm{V_\e}^2_{L^2(Y_\e,m_\e)}= o\left(\mc{E}^{(\e)}(V_\e)\right)\label{hp_o_eginefunctions} 
\end{equation}
then 
\begin{align}
&\mc{E}^{(\e)}(\phi_0-\phi_\e)= \mc{E}^{(\e)}(V_\e)+o\left(\mc{E}^{(\e)}(V_\e)\right) 
\quad \text{ as } \e \to 0^+,\label{eq_conv_eigenfunctions_precisise_o}\\
&\norm{\phi_0-\phi_\e}^2_{L^2(Y_\e,m_\e)}=o\left(\mc{E}^{(\e)}(V_\e)\right) \quad \text{ as } \e \to 0^+. \label{eq_conv_eigenfunctions_L2}
\end{align}
\end{proposition}
\begin{proof}
Since $\la_\e$ is simple, thanks to   \eqref{eq_eigenfunctions} and \eqref{hp_phi_e_diri}
\begin{equation}
\phi_\e=\frac{\Pi_\e(\phi_0-V_\e)}{\norm{\Pi_\e(\phi_0-V_\e)}_{L^2(Y_\e,m_\e)}}.
\end{equation}
Hence 
\begin{equation}
\mc{E}^{(\e)}(\phi_\e - \Pi_\e(\phi_0-V_\e))
=\frac{|1-\norm{\Pi_\e(\phi_0-V_\e)}_{L^2(Y_\e,m_\e)}|^2}{\norm{\Pi_\e(\phi_0-V_\e)}^2_{L^2(Y_\e,m_\e)}} 
\mc{E}^{(\e)}(\Pi_\e(\phi_0-V_\e)) 
\end{equation}
and, by \eqref{eq_eigenfunctions} and  \eqref{hp_conv_Y_e_eginefunctions},
\begin{align}
&\norm{\Pi_\e(\phi_0-V_\e)}_{L^2(Y_\e,m_\e)}^2\\
&= \norm{\phi_0-\Pi_\e(\phi_0-V_\e)}^2_{L^2(Y_\e,m_\e)}+\norm{\phi_0}_{L^2(Y_\e,m_\e)}^2
-2 \ps{L^2(Y_\e,m_\e)}{\phi_0-\Pi_\e(\phi_0-V_\e)}{\phi_0}  \\
&=1+O(\norm{V_\e}_{L^2(Y_\e,m_\e)}).
\end{align}
It follows that 
\begin{equation}
\mc{E}^{(\e)}(\phi_\e - \Pi_\e(\phi_0-V_\e))=O(\norm{V_\e}^2_{L^2(Y_\e,m_\e)}) \quad \text{ as } \e \to 0^+.
\end{equation}
Hence by  \eqref{eq_eigenfunctions_E} we have proved \eqref{eq_conv_eigenfunctions_precisise}.
Furthermore, thanks to \eqref{hp_o_eginefunctions} and \eqref{eq_eigenfunctions_E},
\begin{align}
&\mc{E}^{(\e)}(\phi_0-\phi_\e) \\
&= \mc{E}^{(\e)}(V_\e) + \mc{E}^{(\e)}(\phi_\e-\Pi_\e(\phi_0-V_\e)) 
+2\mc{E}^{(\e)}(\phi_\e-\Pi_\e(\phi_0-V_\e),\Pi_\e(\phi_0-V_\e)-\phi_0) +o\left(\mc{E}^{(\e)}(V_\e)\right)\\
&=\mc{E}^{(\e)}(V_\e) +o\left(\mc{E}^{(\e)}(V_\e)\right) \quad \text{ as } \e \to 0^+,
\end{align}
which proves \eqref{eq_conv_eigenfunctions_precisise_o}. Finally, thanks to \eqref{eq_eigenfunctions} and \eqref{hp_o_eginefunctions}, we can prove \eqref{eq_conv_eigenfunctions_L2} in a similar manner. 
\end{proof}

\subsection{Some additional results to quantify $\la_\e - \la_0$.} \label{subsection_add_results} To quantify more explicitly the rate of convergences of simple eigenvalues provided by  Theorem \ref{theo_exp_1_order}, in many situations results like 
\cite[Proposition 4.5, Proposition 6.5]{FNOS_AB_12} or \cite[Lemma 3.1, Lemma 3.5]{FLO_holes_Neumann} may be 
useful. Indeed they can be used to compute the order of infinitesimal of  $\norm{V_\e}_{L^2(Y_\e,m_\e)}$ or
$\mc{E}^{(\e)}(V_\e)$ as $\e \to 0^+$.
In this subsection we  generalize to our abstract setting the afford mentioned results.
The first one is a characterization of the minimal value of the functional $J_\e$. 
Its interest lays in the fact that if $\norm{L_\e}_{(\mc{F}_\e)^*} \to 0^+$ as $\e \to 0^+$ in a controlled way then the study of the  asymptotic behavior as  of $\mc{E}^{(\e)}(V_\e)$ (thus also of $\norm{V_\e}_{L^2(Y_\e,m_\e)}$) can be  reduced to the study of  asymptotic behavior of $J_\e(V_\e)$. 
Indeed, testing \ref{eq_Ve} with $V_\e$, by Proposition \ref{prop_poin},
\begin{equation}
\la_{1,\e} \norm{V_\e}^2_{L^2(Y_\e,m_\e)} \le \mc{E}^{(\e)}(V_\e) = 2(J_\e(V_\e)+L_\e(V_\e)) 
 \le 2(|J_\e(V_\e)|+\norm{L_\e}_{(\mc{F}_\e)^*}(\mc{E}^{(\e)}(V_\e))^\frac{1}{2}).  
\end{equation}
The advantage of this approach is that competitors may be used to obtain information about the infinitesimal order of $J_\e(V_\e)$, see for example \cite[Proposition 4.3]{FNOS_AB_12}.

\begin{proposition}\label{prop_Ee_equiv}
For any $\e \in [0,1]$
\begin{align}\label{eq_Ee_equiv}
J_\e(V_\e)=-\frac{1}{2} \sup_{w \in Z_\e\setminus \{0\}}
\frac{\left(\mc{E}^{(\e)}(\phi_0,w) - L_\e(w)\right)^2}
{\mc{E}^{(\e)}(w)}+\frac{1}{2}\mc{E}^{(\e)}(\phi_0) -L_\e(\phi_0).
\end{align}
Moreover, if $\phi_0\in Z_\e$, then
\begin{align}\label{eq_Ee_equiv2}
J_\e(V_\e)=-\frac{1}{2} \sup_{w\in Z_\e \setminus\{0\}} \frac{L_\e(w)^2}{\mc{E}^{(\e)}(w)}.
\end{align}
\end{proposition}

\begin{proof}
In view of \eqref{prob_J_min} it follows
\begin{align}\label{proof:prop_Ee_equiv_1}
J_\e (V_\e)&=\inf_{w \in Z_\e} J_\e(w+\phi_0) =\inf_{w \in Z_\e}\left(\inf_{t \in \R}J_\e(tw+\phi_0)\right).
\end{align}
By the definition of $J_\e$, we get
\begin{align}
J_\e(tw+\phi_0)&=\frac{1}{2}\mc{E}^{(\e)}(tw+\phi_0)-L_\e(tw+\phi_0)\\
&= \frac{t^2}{2} \mc{E}^{(\e)}(w) + \frac{1}{2} \mc{E}^{(\e)}(\phi_0) + t \mc{E}^{(\e)}(w,\phi_0) - t L_\e(w) - L_\e (\phi_0)
\end{align}
implying that for any $w \in Z_\e\setminus \{0\}$
\begin{align}
\inf_{t \in \R}J_\e(tw+\phi_0) = &\frac{1}{2}\left(\frac{\mc{E}^{(\e)}(w,\phi_0)-L_\e(w)}{\mc{E}^{(\e)}(w)}\right)^2\mc{E}^{(\e)}(w)+ \frac{1}{2}\mc{E}^{(\e)}(\phi_0)\\
&-\frac{\mc{E}^{(\e)}(w,\phi_0)-L_\e(w)}{\mc{E}^{(\e)}(w)}  \mc{E}^{(\e)}(w,\phi_0) + \frac{\mc{E}^{(\e)}(w,\phi_0)-L_\e(w)}{\mc{E}^{(\e)}(w)} L_\e(w) - L_\e (\phi_0).
\end{align}
Hence, \eqref{eq_Ee_equiv} follows from \eqref{proof:prop_Ee_equiv_1}.

In the particular case $\phi_0 \in Z_\e$ we have that 
\begin{align*}
 J_\e (V_\e)=\inf_{w \in Z_\e} J_\e(w)& =\inf_{w \in Z_\e}\inf_{t \in \R}J_\e(tw)\\
 &=\inf_{w \in Z_\e}\inf_{t \in \R} \left(\frac{t^2}{2} \mc{E}^{(\e)}(w) -t L_\e(w)\right)
\end{align*}
 and so 
\begin{align}
 J_\e(V_\e)=-\frac{1}{2} \sup_{w\in Z_\e \setminus\{0\}} \frac{L_\e(w)^2}{\mc{E}^{(\e)}(w)}.
\end{align}
\end{proof}

The next proposition provides a sufficient condition  to sharpen the asymptotic expansion in \eqref{eq_asymptotic_eigenvlaues_1_order}.
Furthermore in many  applications is  easier to  compute the infinitesimal order of  $\mc{E}(V_\e)$ instead of $\norm{V_\e}_{L^2(Y_\e,m_\e)}$, especially when blow-up arguments are involved,  see for example \cite{FLO_holes_Neumann,FNOS_AB_12}. 
\begin{proposition}\label{prop_norm_L2_o}
Assume  that  the following holds.
\begin{itemize}
\item[(i)] Let $\e_k \to 0^+$ and let  $\{u_k\}_{k\in \mb{N}}$ with 
$u_k \in \mc{F}_{\e_k}$,  $\norm{u_{k_n}}_{L^2(Y_{\e_k},m_{\e_{k_n}})}=1$  and  $\mc{E}^{(\e_k)}(u_k) \le C$,
for some constant $C>0$ that does not depends on $k$. Then  there exists $u \in \mc{F}_0$ and a subsequence $\{u_{k_n}\}_{n \in \mathbb{N}}$  such that:
\begin{itemize}
\item[(i.1)] $\mc{E}^{(\e_{k_n})}(u_{k_n},v) \to \mc{E}^{(0)}(u,v)$ for any  $v \in \mc{F}_0$,
\item[(i.2)] $\norm{u}_{L^2(X,m)}=1$.
\end{itemize}
\item [(ii)] For any $u \in \mc{F}_0$ there exists $\{u_n\}_{n\in \mathbb{N}} \subset \mc{F}_0$ such that:  
\begin{itemize}
\item[(ii.1)]     $u_n\rightharpoonup u$ weakly in $\mc{F}_0$ as $n \to \infty$,
 \item[(ii.2)]     for any $\e_0\in (0,1]$ there exists  $n_0 \in \mb{N}$  such that   $u_n \in Z_\e$  for any $n\ge n_0$ and any $\e \in (0,\e_0]$,
 \item[(ii.3)]     $L_\e(u_n)=0$ for any  $n \ge n_0 $ and any  $\e \in (0,\e_0]$.
\end{itemize}
\end{itemize}
 Then
\begin{align}
    \norm{V_\e}^2_{L^2(Y_\e,m_\e)}=o(\mc{E}^{(\e)}(V_\e)) \quad as\ \e \to 0.
\end{align}
\end{proposition}
\begin{proof}
Suppose by contradiction that there exists a subsequence of $\{\e_k\}_{k \in \mathbb{N}}$ such that 
\begin{align}
\norm{V_{\e_k}}^2_{L^2(Y_{\e_k},m_{\e_k})}\geq C \mc{E}^{(\e_k)}(V_{\e_k})
\end{align}
where $C>0$ is a constant not depending on $k$. Define
\begin{align*}
W_k:= \frac{V_{\e_k}}{\norm{V_{\e_k}}_{L^2(Y_{\e_k},m_{\e_k})}}.
\end{align*}
By assumption, there exists $W\in \mc{F}_0$ with  $\norm{W}_{L^2(X,m)}=1$ and such that, up to a subsequence,
\begin{align}
\mc{E}^{(\e_k)}(W_k,v)\to \mc{E}^{(0)}(W,v) \quad \text{ for any } v \in \mc{F}_0.
\end{align}
Let  $v\in \mc{F}_0$ and let  $\{v_n\}\subset \mc{F}_0$ be as in $({\rm{ii}})$. Then for any $k_0 \in \mathbb{N}$ there exists a $n_0 \in \mathbb{N}$ such that 
\begin{align}
\mc{E}^{(\e_k)}(W_k,v_n)=L_\e(v_n)=0 \quad \text{ for any } n \ge n_0 \text{ and any } k \ge k_0.   
\end{align}
Hence, passing to the limit as $k \to \infty$, we obtain 
\begin{align}
\mc{E}^{(0)}(W,v_n)=0 \quad \text{ for any }   n \ge n_0.
\end{align}
Passing to the limit as $n\to \infty$, we conclude that 
\begin{align}
\mc{E}^{(0)}(W,v)=0 \quad \text{ for any } v \in \mc{F}_0.
\end{align}
Testing the above equation with  $W$, we get $\mc{E}^{(0)}(W)=0$, which is a contradiction in view of Proposition \ref{prop_poin}.
\end{proof}

\subsection{A sufficient condition for spectral stability}\label{subsec_spectral_stability}
In this subsection we provide a general criterion to prove that  assumption \eqref{hp_limit_simple} holds, that is, that there is  spectral stability. It is based on the following result, we refer to \cite[Corollaries XI 9.3, XI 9.4]{DSJ_book} for a proof.
\begin{theorem}\label{theo_eigen_estimate}
Let $R_1, R_2$ be linear bounded, positive, compact, self-adjoint operators on a Hilbert space $\mc{H}$. Let
$\{\mu_{n,i}\}_{n \in \mathbb{N}\setminus\{0\}}:=\sigma(R_i)\setminus\{0\}$ for $i=1,2$. Then 
\begin{equation}
|\mu_{n,1}-\mu_{n,2}| \le \norm{R_1-R_2}_{\mathcal{L}(H)}\quad  \text{ for any } n \in \mb{N}\setminus\{0\}.
\end{equation}
\end{theorem}
The proof of the following theorem is inspired by \cite[Proposition 6.1.7]{BB_book} and this approach to spectral stability is inspired by \cite[Lemma 1.1]{RT_qualitative_stability}.
\begin{theorem}\label{theo_spectral_stability}
Suppose that 
\begin{equation}\label{hp_measure_to_0}
{m_\e=m \text{ for any } \e \in (0,1]} \text{ and  } m(X\setminus Y_\e) \to 0^+ \text{ as } \e\to 0^+.
\end{equation}
Let $R_\e$ be as in \eqref{def_Re}. Assume that for any $\e \in [0,1]$ there exists a bounded linear operator 
\begin{equation}
P_\e:\mc{Z}_0 \to \mc{Z}_\e
\end{equation}
and a bounded  linear operator
\begin{equation}
E_\e:Z_\e \to \mc{Z}_0
\end{equation}
such that $(E_\e u)_{|_{Y_\e}}=u$, $P_\e E_\e u=u$ for any $u \in Z_\e$ and 
\begin{align}
&\lim_{\e  \to 0^+}(P_\e w,v)_{L^2(Y_\e,m)}=(w,v)_{L^2(X,m)} \quad \text{ for any }v,w \in \mc{Z}_0,\label{hp_limit_P}\\
&\ps{L^2(Y_\e,m)}{R_\e P_\e v}{w_{|_{Y_\e}}}=\ps{L^2(Y_\e,m)}{v_{|_{Y_\e}}}{R_\e P_\e w} \quad \text{ for any }v,w \in \mc{Z}_0.\label{hp_RP_auto}
\end{align}
Furthermore, let $(\mc{E}, \mc{F})$ be a positive, densely defined, closed bilinear form and $H$ the correspondent densely defined, positive, self-adjoint operator. Assume that $H^{-1}$ is a well-defined, bounded and compact operator and that $\mc{F} \subset \mc{Z}_0$. 
Finally let
\begin{equation}
\widetilde{R}_\e: \mc{Z}_0 \to  \mc{Z}_0, \quad \widetilde{R}_\e:= E_\e R_\e P_\e
\end{equation}
and suppose   $\widetilde{R}_\e:\mc{Z}_0 \to \mc{F}$ are well-defined and  equibounded, that is there exists a constant $C$, that does not depend on $\e$, such that
\begin{equation}\label{ineq_equi_bounded}
\mc{E}(\widetilde{R}_{\e}u) \le C \norm{u}^2_{L^2(X,m)} \quad \text{ for any } u \in Z_0 \text{ and } \e \in [0,1].
\end{equation}
We also assume that  if $\widetilde{R}_\e u \rightharpoonup w$ weakly in $\mc{F}$ then $w \in Z_0$  and 
\begin{equation}\label{hp_convergence_spectral_stab_bilinear}
\lim_{\e \to 0^+}\mc{E}^{(\e)}(R_\e P_\e u, v)= \mc{E}^{(0)}(w, v), \quad \text{ for any }  v \in Z_0.
\end{equation}
Then for any $n \in \mathbb{N}$
\begin{equation}\label{eq_limit_eigen_stability}
\lim_{\e \to 0^+} \la_{\e,n}=\la_{0,n}.
\end{equation}
\end{theorem}
\begin{proof}
For the sake of simplicity,  we divide the proof in three steps.

\textbf{Step 1.} We have that
\begin{equation}\label{eq_same_spectrum}
    \sigma(\widetilde{R}_\e)=  \sigma(R_\e) \cup \{0\} \quad \text{ or }  \quad  \sigma(\widetilde{R}_\e)=  \sigma(R_\e).
\end{equation}
Indeed, if $R_\e \phi_\e= \la_\e \phi_\e$, then 
\begin{equation}
\widetilde{R}_\e E_\e \phi_\e= E_\e R_\e P_\e E_\e \phi_\e=  E_\e R_\e  \phi_\e =  \la_\e E_\e \phi_\e.
\end{equation}
Since the kernel of $E_\e$ is trivial, it follows that  $\sigma(R_\e) \subseteq \sigma(\widetilde{R}_\e)$.
On the other hand, if $\widetilde{R}_\e \phi_\e= \la_\e \phi_\e$, and $\la_\e \neq 0$, then $P_\e \phi_\e \neq 0$ and
\begin{equation}
 E_\e R_\e( P_\e \phi_\e)=  \la_\e E_\e (P_\e  \phi_\e).
\end{equation}
Hence   $\la_\e $ is a eigenvalue of $R_\e$ since $E_\e$ is injective. We conclude that  \eqref{eq_same_spectrum} holds.

\textbf{Step 2.} For any $u \in \mc{Z}_0$ we have that 
\begin{equation} \label{limit_L2_strong}
 \lim_{\e \to 0^+} \norm{\widetilde{R}_\e u-R_0 u}_{L^2(X,m)}=0.   
\end{equation} 
By \eqref{ineq_equi_bounded} and Proposition \ref{prop_compact_embedding}, for any $u \in \mc{Z}_0$ there exists a sequence $\widetilde{R}_{\e_{n}} u \rightharpoonup w$ weakly  in $\mc{F}$ as $n\to \infty$, for some $w \in \mc{F}$.
Then by \eqref{hp_limit_P}, and \eqref{hp_convergence_spectral_stab_bilinear}  for any $v \in Z_0$
\begin{equation}
(u,v)_{L^2(X,m)}=\lim_{n \to \infty}(P_{\e_n} u,v)_{L^2(Y_\e,m)}=\lim_{\e \to 0^+}\mc{E}^{(\e)}(R_\e P_\e u, v)= \mc{E}^{(0)}(w, v).
\end{equation}
Hence $\widetilde{R}_{\e_{n}} u \to R_0 u$,  strongly in $L^2(X,m)$ as $n\to \infty$, thanks to Proposition \ref{prop_compact_embedding}.
By the Urysohn Subsequence Principle we conclude that  \eqref{limit_L2_strong} holds.

\textbf{Step 3.}  We claim  that 
\begin{equation} \label{limit_operatorial_strong}
 \lim_{\e \to 0^+} \norm{\widetilde{R}_\e-R_0}_{\mc{L}(L^2(X,m))}=0.   
\end{equation}

By the compactness of the operator $\widetilde{R}_\e-R_0$, for any $\e \in [0,1]$ there exists a $f_\e \in L^2(X,m)$ with $\norm{f_\e}_{L^2(X,m)}=1$ such that 
\begin{equation}
\norm{\widetilde{R}_\e-R_0}_{\mc{L}(L^2(X,m))}=\norm{\widetilde{R}_\e f_\e-R_0f_\e}_{L^2(X,m)}.
\end{equation}
Hence, there exists a sequence  $f_{\e_n} \rightharpoonup f$ weakly for some $f \in L^2(X,m)$. By compactness of $R_0$ it follows that $R_0 f_{\e_n} \to R_0 f $ strongly in $L^2(X,m)$, as $n \to \infty$.
By \eqref{ineq_equi_bounded}, up to a subsequence, $\widetilde{R}_{\e_n}f_{\e_n} \to g$  strongly in  $L^2(X,m)$ for some $g\in \mc{F}$.

For any $h \in \mc{Z}_0$, thanks to \eqref{hp_RP_auto},
\begin{align}
&\ps{L^2(X,m)}{\widetilde{R}_{\e_n}f_{\e_n}}{h}=\ps{L^2(Y_\e,m)}{f_{\e_n}}{R_{\e_n} P_{\e_n} h}+ \ps{L^2(X\setminus Y_\e,m)}{\widetilde{R}_{\e_n}f_{\e_n}}{h}\\
&=\ps{L^2(X,m)}{f_{\e_n}}{\widetilde{R}_{\e_n}h} - \ps{L^2(X\setminus Y_\e,m)}{f_{\e_n}}{\widetilde{R}_{\e_n}h} + \ps{L^2(X\setminus Y_\e,m)}{\widetilde{R}_{\e_n}f_{\e_n}}{h}. 
\end{align}
Thanks to the absolute continuity of the integral, 
\begin{equation}
\ps{L^2(X\setminus Y_\e,m)}{f_{\e_n}}{\widetilde{R}_{\e_n}h}\to 0^+, \quad \ps{L^2(X\setminus Y_\e,m)}{\widetilde{R}_{\e_n}f_{\e_n}} {h} \to 0^+, \quad \text{ as } n \to \infty.
\end{equation}
By \eqref{limit_L2_strong}
\begin{equation}
\ps{L^2(X,m)}{g}{h}=\lim_{n\to \infty}\ps{L^2(X,m)}{\widetilde{R}_{\e_n}f_{\e_n}}{h}=\ps{L^2(X,m)}{f}{R_0 h}=\ps{L^2(X,m)}{R_0 f}{h}.
\end{equation}
Hence  $g=R_0f$. By Proposition \ref{prop_compact_embedding}, we conclude that the $\widetilde{R}_{\e_n}f_{\e_n} \to R_0f$ strongly in $L^2(X,m)$. 
By Urysohn's Subsequence Principle we conclude that \eqref{limit_operatorial_strong} holds.

In conclusion \eqref{eq_limit_eigen_stability} follows from Theorem \ref{theo_eigen_estimate} and \eqref{eq_same_spectrum}.
\end{proof}

\begin{remark}
In the applications of Theorem \ref{theo_exp_1_order} presented in Section \ref{sec_applications}, we will use either Theorem \ref{theo_spectral_stability} or min-max arguments to show that \eqref{hp_limit_simple} holds.
\end{remark}

\section{Applications}\label{sec_applications}

In the present section we discuss some applications of Theorem \ref{theo_exp_1_order}.
\subsection{A Robin-to-Neumann problem} In what follows, let $(M,g)$ be a complete Riemannian manifold of dimension $\textnormal{dim}(M)=n\geq 2$ and $\Omega\subset M$ a smooth bounded domain. We adopt the convention that $\D$ is the negative definite Laplace-Beltrami operator, so to have $\D=\frac{d^2}{dx^2}$ in $\R$. Moreover, we denote by $\textnormal{dv}$ the Riemannian volume density of $M$ and by $\textnormal{da}$ the $(n-1)$-dimensional area element.
Fix $\e\in(0,1]$ and c the following Robin eigenvalue problem
\begin{align}\label{Prob:Robin}
\begin{cases}
-\D u = \la u, & \inn \Omega,\\ 
\frac{\partial u}{\partial \nu}=-\e u, & \onn \partial \Omega,
\end{cases}
\end{align}
where $\nu$ is the outward pointing unit normal to $\partial \Omega$ and $\lambda_\e$ is a simple eigenvalue. Our aim is to show that we are able to predict, at the first order, the rate of convergence of the simple eigenvalues of \eqref{Prob:Robin} to that of the limit problem
\begin{align}\label{Prob:LimitNeumann0}
\begin{cases}
-\D u = \lambda u, & \inn \Omega, \\ 
\frac{\partial u}{\partial \nu}=0, & \onn \partial \Omega
\end{cases}
\end{align}
using the asymptotic expansion provided by Theorem \ref{theo_exp_1_order}. 
The complementary case, that is, the limit as $\e \to \infty$ was studied in detail in the very recent paper \cite{O_Robin}.

Since the first eigenvalue of the boundary value problem \eqref{Prob:LimitNeumann0} is 0, as stressed in Section \ref{sec_assump_preli} it is equivalent to consider the following family of problems 
\begin{align}\label{Prob:LimitNeumann}
\begin{cases}
-\D u +u= \lambda u, & \inn \Omega,
\\ \frac{\partial u}{\partial \nu}=-\e u, & \onn \partial \Omega,
\end{cases}
\end{align}
for $\e \in [0,1]$. Let $(\mc{E}^{(\e)},\mathcal{F}_\e)$ be the Dirichlet form associated to  problem \eqref{Prob:LimitNeumann}
\begin{align}
    \mc{E}^{(\e)}(u,v):= \int_\Omega \left[g(\nabla u, \nabla v)+uv \right] \dvol + \e \int_{\partial \Omega} uv \da
\end{align}
with domain $\mc{F}_\e:= H^1(\Omega)$.
Hence, fixed $\lambda_0$ a simple eigenvalue of problem \eqref{Prob:LimitNeumann} with $\e=0$ and $\phi_0$ a  associated normalized eigenfunction  we have
\begin{align*}
    \mc{E}^{(\e)}(\phi_0,u)=\lambda_0 \int_\Omega \phi_0 u \dvol + \e \int_{\partial \Omega} \phi_0 u \da\quad \text{ for any } u \in H^1(\Omega).
\end{align*}
Denoting by $L_\e :H^1(\Omega)\to \R$ the linear and continuous operator
\begin{align}
    L_\e (u):= \e \int_{\partial \Omega} \phi_0 u \da
\end{align}
and by $J_\e:H^1(\Omega)\to \R$ the functional
\begin{align}
    J_\e(u): & = \frac{1}{2}\mc{E}^{(\e)}(u)-L_\e (u)\\
             & = \int_\Omega \left[|\nabla u|^2 + u^2 \right] \dvol + \e \int_{\partial \Omega} u^2 \da- \e \int_{\partial \Omega} \phi_0 u \da,
\end{align}
by Proposition \ref{prob_J_min} we obtain the existence of the (unique) minimum $V_\e \in H^1(\Omega)$ of the minimization problem
\begin{align}
    \inf\{J_\e (u)\ :\ u\in H^1(\Omega)\}.
\end{align}
Now we show that  $V_\e\xrightarrow[]{\e\to 0}0$ in $H^1(\Omega)$. Indeed, by Proposition \ref{prob_J_min}, $V_\e$ satisfies
\begin{align}
    \mc{E}^{(\e)}(V_\e,u)=L_\e(u) \quad \text{ for any } u \in  H^1(\Omega),
\end{align}
i.e.
\begin{align}\label{Eq:RobinApp1}
    \int_\Omega \left[g(\nabla V_\e, \nabla u)+V_\e u\right] \dvol + \e \int_{\partial \Omega} V_\e u \da = \e \int_{\partial \Omega} \phi_0 u \da \quad \text{ for any } u \in H^1(\Omega).
\end{align}
Testing \eqref{Eq:RobinApp1} with $V_\e$ itself, it follows that
\begin{align}
    \int_\Omega \left[|\nabla V_\e |^2+V_\e^2 \right] \dvol + \e \int_{\partial \Omega} V_\e^2 \da = \e \int_{\partial \Omega} \phi_0 V_\e \da
\end{align}
implying
\begin{align}\label{Eq:RobinApp2}
    \hnorm{V_\e}{1}{\Omega}^2 &\leq  \e \int_{\partial \Omega} \phi_0 V_\e \dvol\\
    & \leq \e \lnorm{\phi_0}{2}{\partial \Omega} \lnorm{V_\e}{2}{\partial \Omega}.
\end{align}
As a consequence of the following classical inequality, we get
\begin{align}
    \hnorm{V_\e}{1}{\Omega}\xrightarrow[]{\e\to 0}0.
\end{align}

\begin{proposition}
Let $(M,g)$ be a complete Riemannian manifold and $\Omega \subset M$ a smooth domain. Then, there exists a positive constant $C_0>0$ so that for every $u\in H^1(\Omega)$
\begin{align}\label{Eq:Trace}
\int_{\partial \Omega} |u|^2 \da\leq C_0 \int_\Omega \left[|\nabla u|^2+u^2 \right] \dvol.
\end{align}
\end{proposition}
\noindent If the dimension of $M$ is at least $3$, then the results is contained in \cite[Proposition 1.1]{LiZhu}. Otherwise, in any dimension, it can be deduce from classical trace results in smooth bounded domain in $\R^N$, see \cite[Theorem 18.1]{L_book_2_edi}, together with the choice of a finite cover of local charts whose closures smoothly intersect the boundary $\partial \Omega$.
\smallskip

Let $\lambda_{n,\e}$ be the $n$-th eigenvalue of problem \eqref{Prob:Robin} for any $\e \in [0,1]$.
\begin{align}
    \lambda_{n,\e}= \inf_{\underset{\dim(A)=n}{A\subset H^1(\Omega)}} \sup_{\underset{\lnorm{u}{2}{ \Omega }=1}{u\in A}} \left(\int_\Omega \left[|\nabla u|^2 + u^2\right] \dvol + \e \int_{\partial \Omega} u^2 \da \right),
\end{align}
where the infimum is taken over all the $n$-dimensional subspaces $A$ of $H^1(\Omega)$. If $A\subset H^1(\Omega)$ is a fixed $n$-dimensional vector space, then by \eqref{Eq:Trace} one has
\begin{align}
    \int_\Omega \left[|\nabla u|^2 + u^2\right] \dvol & \leq  \int_\Omega \left[|\nabla u|^2 + u^2\right] \dvol + \e \int_{\partial \Omega} u^2 \da\\
    &\leq (1+\e C) \int_\Omega \left[|\nabla u|^2 + u^2\right] \dvol \quad \quad \quad \quad \text{ for any } u \in A,
\end{align}
obtaining
\begin{align}
    \sup_{\underset{\lnorm{u}{2}{ \Omega }=1}{u\in A}} \int_\Omega \left[|\nabla u|^2 + u^2\right] \dvol &\leq \sup_{\underset{\lnorm{u}{2}{ \Omega }=1}{u\in A}} \left(\int_\Omega \left[|\nabla u|^2 + u^2\right] \dvol + \e \int_{\partial \Omega} u^2 \da \right)\\
    &\leq (1+\e C) \sup_{\underset{\lnorm{u}{2}{ \Omega }=1}{u\in A}} \int_\Omega \left[|\nabla u|^2 + u^2\right] \dvol.
\end{align}
Taking the infimum over all the $n$-dimensional subspaces $A\subset H^1(\Omega)$, it follows
\begin{align}
    \lambda_{0,n}\leq \lambda_{\e,n}\leq (1+\e C) \lambda_{0,n}
\end{align}
that implies
\begin{align}
    \lim_{\e\to 0} \lambda_{\e,n}=\lambda_{0,n}.
\end{align}
As a consequence, by Theorem \ref{theo_exp_1_order}, we get
\begin{align}
    \la_{\e,n}-\la_{0,n}=\la_{0,n} \int_{\Omega} V_\e \phi_0 \dvol +O\left(\lnorm{V_\e}{2}{\Omega}^2 \right)\quad \textnormal{as } \e \to 0.
\end{align}

We study the asymptotic behavior of $V_\e$ in order to obtain a more explicit description of the eigenvalue variation. To this aim, let us define
\begin{align*}
    \widetilde{V}_\e (x):= \frac{V_\e(x)}{\e}.
\end{align*}
Firstly, we observe that by \eqref{Eq:RobinApp1}
\begin{align}\label{Eq:RobinApp3}
    \int_\Omega \left[g(\nabla \widetilde{V}_\e, \nabla u)+\widetilde{V}_\e u\right] \dvol + \e \int_{\partial \Omega} \widetilde{V}_\e u \da = \int_{\partial \Omega} \phi_0 u \da
\end{align}
for every $u\in H^1(\Omega)$, implying, by testing the equation with $u=\widetilde{V}_\e$,
\begin{align}
    \int_\Omega \left[ \left|\nabla \widetilde{V}_\e\right|^2+\widetilde{V}_\e^2\right] \dvol \leq \int_{\partial \Omega} \phi_0 \widetilde{V}_\e \da.
\end{align}
By H\"older inequality and \eqref{Eq:Trace}, we get
\begin{align}
    \hnorm{\widetilde{V}_\e}{1}{\Omega}^2 &\leq \lnorm{\phi_0}{2}{\partial \Omega} \lnorm{\widetilde{V}_\e}{2}{\partial \Omega}\\
    &\leq C \lnorm{\phi_0}{2}{\partial \Omega} \hnorm{\widetilde{V}_\e}{1}{\Omega}
\end{align}
i.e.
\begin{align}
    \hnorm{\widetilde{V}_\e}{1}{\Omega} \leq C \lnorm{\phi_0}{2}{\partial \Omega},
\end{align}
thus  the family $\{\widetilde{V}_\e\}_\e$ is bounded in $H^1(\Omega)$. Whence, there exists a sequence $\e_n \to 0^+$ as $n\ to \infty$ and a function $\widetilde{V}\in H^1(\Omega)$ such that $\widetilde{V}_{\e_n}\xrightharpoonup{H^1(\Omega)} \widetilde{V}$ weakly  in $H^1(\Omega)$ as $n\to \infty$. 

Passing to the limit in \eqref{Eq:RobinApp3}, it follows that 
\begin{align}\label{Eq:RobinApp4}
    \int_\Omega \left[g(\nabla \widetilde{V}, \nabla u)+\widetilde{V} u\right] \dvol = \int_{\partial \Omega} \phi_0 u \da
\end{align}
for every $u\in H^1(\Omega)$.

\begin{remark}\label{Rmk:RobinApp1}
    We stress that \eqref{Eq:RobinApp4} has at most one solution. Indeed, if $\widetilde{V}$ and $\widetilde{W}$ are two functions satisfying the equation above, then the function $\widetilde{U}:=\widetilde{V}-\widetilde{W}$ satisfies
    \begin{align*}
        \int_\Omega \left[g(\nabla \widetilde{U}, \nabla u)+\widetilde{U} u\right] \dvol = 0
    \end{align*}
    for every $u\in H^1(\Omega)$, implying $\widetilde{U}=0$ and hence $\widetilde{V}=\widetilde{W}$.
\end{remark}

As a consequence of Remark \ref{Rmk:RobinApp1}, it follows that every subsequence $\{\e_n\}_n$ so that $\{\widetilde{V}_{\e_n}\}_n$ converges weakly must have $\widetilde{V}$ as limit function. In particular, by the Urysohn Subsequence Principle it follows that $\widetilde{V}_{\e}\xrightharpoonup{H^1(\Omega)} \widetilde{V}$ weakly  in $H^1(\Omega)$ and not only along the sequence $\{\e_n\}_n$. By \eqref{Eq:RobinApp3} we get
\begin{align}
    \lim_{\e \to 0^+} \hnorm{\widetilde{V}_\e}{1}{\Omega}^2 &= \lim_{\e \to 0^+} \left[ \int_{\partial \Omega} \phi_0 \widetilde{V}_\e \da - \e \int_{\partial \Omega} |\widetilde{V}_\e|^2 \da \right]\\
    &= \int_{\partial \Omega} \phi_0 \widetilde{V} \da\\
    &= \hnorm{\widetilde{V}}{1}{\Omega}^2,
\end{align}
where the last equality follows by \eqref{Eq:RobinApp4}. We also observe that, by the very definition of $\widetilde{V}_\e$, this exactly means that
\begin{align}\label{Eq:RobinApp5}
    \hnorm{V_\e}{1}{\Omega}^2 = O(\e^2)\quad \textnormal{as}\ \e \to 0.
\end{align}
Moreover, since in any Hilbert space the weak convergence together with the convergence of the norm implies the strong convergence, we get that
\begin{align}
    \widetilde{V}_\e \xrightarrow{\e \to 0} \widetilde{V} \quad \text{ strongly in } H^1(\Omega).
\end{align}
Lastly, we stress that by \eqref{Eq:RobinApp4}
\begin{align}
    \lim_{\e \to 0} \lambda_0 \int_\Omega \e^{-1} V_\e \phi_0 \dvol &=\lambda_0 \int_\Omega \widetilde{V} \phi_0 \dvol\\
    &= \int_\Omega \left[ g(\nabla \widetilde{V}, \nabla \phi_0)+ \widetilde{V} \phi_0\right] \dvol\\
    &= \int_\Omega |\phi_0|^2 \da
\end{align}
which, together with \eqref{Eq:RobinApp5}, provides the following expression
\begin{align*}
    \lambda_{\e,n}-\lambda_{0,n}=\e \int_{\partial \Omega} |\phi_0|^2 \da + O(\e^2) \quad \textnormal{as}\ \e \to 0.
\end{align*}

\subsection{Conformal transformations}
Let $(M,g)$ be a complete Riemannian manifold and $\Omega\subseteq M$ a compact domain. Fixed a smooth function
\begin{align}
    \Phi:[0,1]\times M &\to \mathbb{R}\\ (\e,p)&\mapsto \Phi_\e(p)
\end{align}
so that $\Phi_\e\to 0$ uniformly in $\Omega$, consider the following family of metrics on $M$
\begin{align}
    g_\e:=e^{2\Phi_\e} g
\end{align}
which are conformal to $g=g_0$.

To every $\e\in [0,1]$ we can associate the following eigenvalue problems
\begin{align}\label{Prob:Conformal}
    \begin{cases}
        -\Delta_\e u = \lambda u, & \textnormal{in}\ \Omega,\\ 
        u =0, & \textnormal{on}\ \partial \Omega,
    \end{cases}
\end{align}
where by $\Delta_\e$ we mean the Laplace-Beltrami operator associated to the metric $g_\e$.

For any fixed $\e\in [0,1]$, the bilinear form associated to the problem \eqref{Prob:Conformal} is $(\mc{E}^{(\e)},\mathcal{F}_\e)$, where $\mathcal{F}_\e=H^1_0(\Omega,\textnormal{dv}_\e)$ and
\begin{align*}
    \mathcal{E}^{(\e)}(u,v):=\int_\Omega g_{\e}(\nabla_\e u, \nabla_\e v)\dvol_\e\quad \text{ for any } u,v\in \mathcal{F}_\e.
\end{align*}
Since $\Omega$ is compact, it is a standard fact (see \cite{He96}) that $H_0^1(\Omega,\textnormal{dv}_\e)$ does not depend on the metric $g_\e$. Hence, in the following we consider the bilinear forms $\mc{E}^{(\e)}$ as acting over the same domain $\mc{F}:=\mc{F}_0$. Clearly, assumptions \eqref{A2}, \eqref{hp_la_1>0} and \eqref{hp_uin_F0_u_in_Fe} are satisfied.

We recall the following useful identities
\begin{align*}
    g_\e^{-1}=e^{-2\Phi_\e} g^{-1}, \quad \quad \nabla_\e = e^{-2\Phi_\e} \nabla \quad \quad \textnormal{and} \quad \quad \dvol_\e = e^{n\Phi_\e} \dvol.
\end{align*}
As a consequence, using the fact that $\Phi_\e$ converges uniformly to $0$ on $\Omega$, we have the following control on the $H^1_0$ norms
\begin{align}\label{Conformal_equiv_norms}
e^{-(n-2)} \norm{u}_{H_0^1(\Omega,\dvol_\e)}  \leq \norm{u}_{H_0^1(\Omega,\dvol)} \leq e^{n-2} \norm{u}_{H_0^1(\Omega,\dvol_\e)}
\end{align}
for $\e$ small enough.

Fix $\lambda_0$ a simple eigenvalue (with associated normalized eigenfunction $\phi_0$) to the limit problem
\begin{align}
\begin{cases}
-\Delta_0 u = \lambda_0 u, & \inn \Omega, \\ 
u=0 & \onn \partial \Omega.
\end{cases}
\end{align}
In particular, for every $\e \in [0,1]$
\begin{align*}
    \mathcal{E}^{(\e)}(\phi_0, u)=\lambda_0 \int_\Omega \phi_0 u \dvol + \int_\Omega g_\e (\nabla_\e \phi_0, \nabla_\e u) \dvol_\e - \int_\Omega g(\nabla \phi_0, \nabla u) \dvol \quad \text{ for any } u\in H_0^1(\Omega),
\end{align*}
showing that condition \eqref{eq_Vehi0_in_Ze} is satisfied, where $L_\e:H^1_0(\Omega)\to \mathbb{R}$ is given by the linear and continuous operator
\begin{align*}
    L_\e(u):=\int_{\Omega} g_\e ( \nabla_\e \phi_0, \nabla_\e u) \dvol_\e - \int_\Omega g(\nabla \phi_0, \nabla u) \dvol.
\end{align*}
Denoting by $J_\e:H^1_0(\Omega)\to \mathbb{R}$ the functional
\begin{align*}
    J_\e(u):&=\frac{1}{2} \mc{E}^{(\e)}(u)- L_\e(u)\\
    &=\frac{1}{2}\int_\Omega g_{\e}(\nabla_\e u, \nabla_\e u)\dvol_\e -\int_{\Omega} g_\e ( \nabla_\e \phi_0, \nabla_\e u) \dvol_\e + \int_\Omega g(\nabla \phi_0, \nabla u) \dvol,
\end{align*}
let $V_\e\in H^1_0(\Omega)$ be the function provided by Proposition \ref{prop_J_min}. We start by showing  that
\begin{align}\label{Conformal:V_e->0}
    \lim_{\e\to 0} ||V_\e||_{L^2(M,\textnormal{dv}_\e)}=0.
\end{align}
By Proposition \ref{prop_J_min}, we have
\begin{align}
    \mc{E}^{(\e)}(V_\e,u)=L_\e(u)\quad \text{ for any } u\in H^1_0(M)
\end{align}
and hence
\begin{align}\label{Conformal_Eq_Ve}
    \int_\Omega g_\e(\nabla_\e V_\e , \nabla_\e u)\dvol_\e &= \int_\Omega g_\e (\nabla_\e \phi_0, \nabla_\e u) \dvol_\e - \int_\Omega g(\nabla \phi_0, \nabla u)\dvol\\
    &=\int_\Omega \left(e^{(n-2)\Phi_\e} -1\right) g(\nabla \phi_0, \nabla u)\dvol
\end{align}
for any $u\in H^1_0(\Omega)$. Testing previous equality with $u=V_\e$, we get
\begin{align}
    \int_\Omega g_\e(\nabla_\e V_\e, \nabla_\e V_\e) \dvol_\e &\leq \int_\Omega \left|e^{(n-2)\Phi_\e} -1\right| |g(\nabla \phi_0, \nabla V_\e)|\dvol\\
    &\leq \max_{\overline{\Omega}}\left| e^{(n-2)\Phi_\e}-1\right| \lambda_0^{\frac{1}{2}} \left(\int_\Omega g(\nabla V_\e, \nabla V_\e) \dvol \right)^\frac{1}{2}\\
    &=\max_{\overline{\Omega}}\left| e^{(n-2)\Phi_\e}-1\right| \lambda_0^{\frac{1}{2}}  \left(\int_\Omega \frac{e^{(n-2)\Phi_\e}}{e^{(n-2)\Phi_\e}} g(\nabla V_\e, \nabla V_\e) \dvol \right)^\frac{1}{2}\\
    &\leq \max_{\overline{\Omega}}\left| e^{(n-2)\Phi_\e}-1\right| \lambda_0^{\frac{1}{2}} \frac{1}{e^{\left[\frac{n-2}{2}\min_{\overline{\Omega}}\Phi_\e\right]}} \left(\int_\Omega g_\e(\nabla_\e V_\e, \nabla_\e V_\e) \dvol_\e \right)^\frac{1}{2} 
\end{align}
that implies 
\begin{align}\label{Test_Conformal}
        \left(\int_\Omega g_\e(\nabla_\e V_\e, \nabla_\e V_\e) \dvol_\e \right)^\frac{1}{2} \leq \max_{\overline{\Omega}}\left| e^{(n-2)\Phi_\e}-1\right|\frac{1}{e^{\left[\frac{n-2}{2}\min_{\overline{\Omega}}\Phi_\e\right]}} \lambda_0^{\frac{1}{2}} 
\end{align}
and, by Proposition \ref{prop_poin},
\begin{align}
    ||V_\e||_{L^2(\Omega,\textnormal{dv}_\e)} & \leq \max_{\overline{\Omega}}\left| e^{(n-2)\Phi_\e}-1\right|  \frac{1}{e^{\left[\frac{n-2}{2}\min_{\overline{\Omega}}\Phi_\e\right]}}\left(\lambda_\e^{-1}\lambda_0\right)^{\frac{1}{2}}\\
    &\leq \left(e^{(n-2)\norm{\Phi_\e}_{L^\infty(\Omega)}}-1\right)\frac{1}{e^{\left[\frac{n-2}{2}\min_{\overline{\Omega}}\Phi_\e\right]}}\left(\lambda_\e^{-1}\lambda_0\right)^{\frac{1}{2}}\\
    &=O\left(e^{(n-2)\norm{\Phi_\e}_{L^\infty(\Omega)}}-1 \right) \quad \quad \textnormal{as}\ \e \to 0\\
    &=O\left(\norm{\Phi_\e}_{L^\infty(\Omega)} \right) \quad \quad \textnormal{as}\ \e \to 0,
\end{align}
since $\norm{\phi_\e}_{L^\infty(\Omega)}\xrightarrow[]{\e \to 0}0$, implying that condition \eqref{hp_limit_E_e_Ve_o} is satisfied.

To prove the stability of the spectrum, i.e. condition \eqref{hp_limit_simple}, we start by observing that
\begin{align*}
    \frac{e^{(n-2) \min_{\overline{\Omega}}\Phi_\e}}{e^{n \max_{\overline{\Omega}}\Phi_\e}} \frac{\int_\Omega g(\nabla u, \nabla u) \dvol}{\int_\Omega u^2 \dvol} \leq \frac{\int_\Omega g_\e (\nabla_\e u, \nabla_\e u)\dvol_\e}{\int_\Omega u^2 \dvol_\e}\leq \frac{e^{(n-2) \max_{\overline{\Omega}}\Phi_\e}}{e^{n\min_{\overline{\Omega}}\Phi_\e}} \frac{\int_\Omega g(\nabla u, \nabla u) \dvol}{\int_\Omega u^2 \dvol}
\end{align*}
Hence with   min-max argument, 
\begin{align*}
    \frac{e^{(n-2) \min_{\overline{\Omega}}\Phi_\e}}{e^{n \max_{\overline{\Omega}}\Phi_\e}} \lambda_{n,0} \leq \la_{n,\e}\leq \frac{e^{(n-2) \max_{\overline{\Omega}}\Phi_\e}}{e^{n\min_{\overline{\Omega}}\Phi_\e}} \lambda_{n,0}
\end{align*}
and so $\lambda_{n,\e}\to \lambda_{n,0}$ as $\e \to 0$ for any $n \in \mb{N}$.

If $\lambda_0$ is simple, by Theorem \ref{theo_exp_1_order} we have that 
\begin{align}
    \lambda_\e -\lambda_0 = \frac{\lambda_0 \int_\Omega \phi_0 V_\e \dvol_\e + O\left( ||V_\e||^2_{L^2 ( \Omega,\textnormal{dv}_\e)} \right)}{\int_\Omega \phi_0^2 \dvol_\e + O\left( ||V_\e||_{L^2 ( \Omega,\textnormal{dv}_\e)}\right)} \quad \quad \textnormal{as}\ \e \to 0.
\end{align}
Noticing that
\begin{align}
    \left|1- \int_\Omega\phi_0^2 \dvol_\e \right|&= \left|\int_\Omega (1-e^{n\Phi_\e}) \phi^2_0 \dvol \right|\\
    &\leq \max_{\overline{\Omega}} \left|1-e^{n\Phi_\e}\right| \int_\Omega \phi_0^2 \dvol\\
    &=\max_{\overline{\Omega}}\left|1-e^{n\Phi_\e}\right|\\
    &=O\left(\norm{\Phi_\e}_{L^\infty(\Omega)} \right) \quad \quad \textnormal{as}\ \e \to 0,
\end{align}
as in Remark \ref{Rmk:Accurate1Exp} it follows that
\begin{align}
\lambda_\e - \lambda_0 &=\left[ \lambda_0 \int_\Omega \phi_0 V_\e \dvol_\e +O\left( \norm{\Phi_\e}_{L^\infty(\Omega)}^2 \right) \right] 
\left( 1+O\left( \norm{\Phi_\e}_{L^\infty(\Omega)} \right)\right)\\
 &=\lambda_0 \int_\Omega \phi_0 V_\e \dvol_\e +O\left( \norm{\Phi_\e}_{L^\infty(\Omega)}^2 \right) \quad \quad \textnormal{as}\ \e \to 0.
\end{align}

Under additional assumption on $\Phi_\e$, we  can obtain a more precise precise result.
Let us suppose that 
\begin{align}\label{conformal_factor_conditions}
\norm{\Phi_\e}_{L^\infty(\Omega)}   =o\left(\sqrt{\e}\right)\quad \textnormal{as}\ \e \to 0,
\end{align}
and
\begin{align}\label{hp_Phi_incremental}
\left|\frac{e^{(n-2)\Phi_\e(x)}-1}{\e}\right|\le  h(x) \quad \text{ for a.e. } x \in \Omega, \text{ and any } \e \text{ small},
\end{align}
where $h \in L^2(\Omega)$. Furthermore, we require that for any $x \in \Omega$ there exists the limit
\begin{align}\label{hp_Phi_precise}
\lim_{\e \to 0^+}\frac{\Phi_\e(x)}{\e}=: \Psi(x)
\end{align}
and that $\Psi \in L^\infty(\Omega)$. 

In view of \eqref{Conformal_equiv_norms}, \eqref{Test_Conformal} and  \eqref{hp_Phi_precise} we have that the family $\left\{\frac{V_\e}{\e}\right\}_{\e \in [0,1]}$ is bounded  in $H^1_0(\Omega)$. In particular, there exists a subsequence $\frac{V_{\e_n}}{\e_n}$ and 
$\widetilde{V} \in H^1_0(\Omega)$ such that $\frac{V_{\e_n}}{\e_n} \rightharpoonup \widetilde{V}$ weakly in $H_0^1(\Omega)$ as $n\to \infty$.

Then by the Dominated Converge Theorem, \eqref{hp_Phi_incremental} and \eqref{hp_Phi_precise} we can pass to the limit in \eqref{Conformal_Eq_Ve}. It follows that $\tilde{V}$ solves the equation
\begin{equation}
\int_{\Omega} g(\nabla \tilde{V}, \nabla u) \dvol = (n-2)\int_{\Omega} \Psi\  g(\nabla \phi_0, \nabla u) \dvol.
\end{equation}
Since the solution to the equation above is unique, by the Urysonh subsequence principle, we can see that $\frac{V_{\e}}{\e} \rightharpoonup \widetilde{V}$ weakly in $H_0^1(\Omega)$ as $\e \to 0^+$.

Finally by the Dominated Converge Theorem 
\begin{align}
\lim_{\e \to 0^+} \lambda_0 \int_\Omega \phi_0 \frac{V_\e}{\e} \dvol_\e
&= \lambda_0 \int_\Omega \phi_0 \widetilde{V} \dvol\\
&= \int_\Omega g(\nabla \widetilde{V}, \nabla \phi_0) \dvol\\
&= (n-2) \int_\Omega \Psi\ g(\nabla \phi_0, \nabla \phi_0) \dvol.
\end{align}
Hence
\begin{align}
    \lambda_\e - \lambda_0 = \e (n-2)\int_\Omega \Psi\ g(\nabla \phi_0, \nabla \phi_0) \dvol + o(\e) \quad \text{ as }\e \to 0^+,
\end{align}
in view of \eqref{conformal_factor_conditions}.

\begin{remark}\label{rem_berg}
Given a family of metrics $\{g_\e\}_{\e \in [0,1]}$, if it is known  a priory that the family of associated eigenvalues $\{\la_{\e,n}\}_{\e \in [0,1]}$ and  eigenfunctions $\{\phi_{\e,n}\}_{\e \in [0,1]}$ are smooth with respect to $\e$, which for example happens whenever $\{g_\e\}_{\e \in [0,1]}$ is analytic with respect to $\e$, then a second order estimate was obtained in \cite{Be73}. See also \cite{ESI08}.
\end{remark}

\subsection{Dirichlet forms}\label{subsect_Diri_forms}
Let $(X,d,m)$ be a locally compact and separable measure metric space, where $m$ is a positive Radon measure defined on the Borel $\sigma$-algebra $\mathcal{B}$ of $X$. Consider $(\mc{E},\mc{F})$ a Dirichlet form associated to the linear, positive, self-adjoint and densely defined operator $H$. Moreover, suppose that $H$ has a discrete spectrum $\{\lambda_n\}_{n \in \mb{N}}$ with associated orthonormal basis of eigenfunctions $\{\phi_n\}_{n \in \mb{N}}$ of $L^2(X,m)$. In view of Remark \ref{Rmk:Accurate1Exp}, it is not restrictive to suppose that $\lambda_1>0$, and so $R:=H^{-1}$ is a well defined, positive, self-adjoint, bounded and compact operator.
Let
\begin{equation}\label{def_C0}
C_c(X):=\{ u: X \to \mb{R}\ :\  u \text{ is continuous with compact support}\}.
\end{equation}
We assume that 
\begin{equation}\label{AS_density}
C_c(X) \cap \mc{F} \text{ is dense in } \mc{F} \text{ with respect to the norm } \mc{E}_1.
\end{equation}
Given a Dirichlet form $(\mc{E}, \mc{F})$ and a function $f\in \mc{F}$, we recall that the $f$-capacity of a set $K\subset X$ is defined as
\begin{align}\label{prob_min_cap}
\text{Cap}_f(K):= \inf \left\{\mc{E}(u) \ :\ u \in \mc{F},\ \widetilde{u}\geq \widetilde{f}\ \text{q.e. in}\ K\right\},
\end{align}
where by $\widetilde{u}$ we mean the quasi-continuous representative of the function $u\in \mc{F}$. We recall that every element of $\mc{F}$ has a quasi-continuous representative with respect to the classical notion of capacity associated to $(\mc{E}, \mc{F})$, see for example \cite{FM_Dirichlet_forms_book}.

In the spirit of \cite{ACF_spectral_stability_AB}, we consider family  $\{K_\e\}_{\e \in [0,1]}$  of compact subsets of $X$ such that:
\begin{itemize}
    \item [(i)] $\capa{}{K_0}:=\capa{1}{K_0}=0$
    \item[(ii)] the family $\{K_\e\}_{\e \in (0,1]}$ concentrates at $K_0$, that is, for every open neighbourhood $U$ of $K_0$ there exists $\e_0\in (0,1]$ so that $K_\e \subset U$ for every $\e<\e_0$.
\end{itemize}
Heuristically, the simpler example of concentrating family of compact sets are shrinking holes but the   assumption that $\{K_\e\}_{\e \in (0,1]}$ concentrates at $K_0$ holds in many other cases. 
For example if $K_\e \to K_0$ in the sense of Hausdorff as $\e \to 0^+$, then it is easy to see that $K_\e$ concentrate at $K_0$. On the other hand concentration of sets is a more general notion, since for example  if $\{K_\e\}_{\e \in (0,1]}$ concentrates at $K_0$ then it also concentrate to any set $\tilde{K_0}$ such that $K_0\subseteq \tilde{K_0}$ while the limit in sense of Hausdorff is unique on compact sets.

Let
\begin{align}
Z_\e:=\{u\in \mc{F}\ :\tilde{u}=0\ \text{q.e. on}\ K_\e\}.
\end{align}
It is easy to see that $\mc{Z}_\e$, the closure of $Z_\e$ with respect to the norm of $L^2(X,m)$, is 
\begin{equation}
\{u\in L^2(X,m)\ :\ u=0\ \text{a.e. on}\ K_\e\}.
\end{equation}
$(\mc{E},Z_\e)$ is still a Dirichlet form, see \cite{FM_Dirichlet_forms_book}, and we denote by $H_\e$ its positive, densely defined (in $\mc{Z}_\e$), self-adjoint operator. 
We may suppose that $R_\e:=H_\e^{-1}$ is a well defined and bounded operator thanks to Remark \ref{Rmk:Accurate1Exp}.  In particular,  since $R_\e$ takes values in $\mc{F}$ by Proposition \ref{prop_compact_embedding} it follows that 
\begin{equation}
 R_\e:\mc{Z}_\e \to \mc{Z}_\e   
\end{equation}
is compact. Hence its spectrum is discrete. Let  $\{\la_{\e,n}\}_{n \in \mathbb{N}}$ be the discrete spectrum of $H_\e$.

Fix a simple eigenvalue $\la:=\la_{n}$ of $(\mc{E},\mc{F})$ and let $\phi$ be a correspondent eigenfunction with $\norm{\phi}_{L^2(X,m)}=1$.
Clearly for any $w \in Z_\e$
\begin{equation}\label{eq_phi0_Dir_form}
\mc{E}(\phi,w)=\la (\phi,w)_{L^2(X,m)}
\end{equation}
and so $L_\e=0$. Let $J_\e$ be as in \eqref{def_J_ef} and consider the function $V_\e \in \mc{F}$ given by Proposition \ref{prop_J_min}. $V_\e$ coincides with the classical capacitary potential $V_\e^\phi$ associated to the $\phi$-capacity $\capa{\phi}{K_\e}$. Indeed $V_\e^\phi$ solves the minimization problem \eqref{prob_min_cap} and so, since $V_\e^\phi=\phi$ on $K_\e$, it solves \eqref{prob_J_min}, which admits a unique solution.
We remark that $V_\e$ satisfies 
\begin{equation}\label{eq_Ve_Dir_form}
\mc{E}(V_\e,w)=0  \quad \text{ for any } w \in Z_\e.
\end{equation}

In order to prove the stability of the spectrum of $H$ and the validity of \eqref{hp_limit_E_e_Ve_o}, we start with the following technical lemma.
\begin{lemma}\label{time_consuming_lemma}
Let $K$ be a compact subset of $X$ and suppose that $\capa{}{K}=0$. Then the set $Z_K:=\{u \in \mc{F}: \text{ such that } K \cap\mathop{\rm{supp}}{u}= \emptyset \}$ is dense in $\mathcal{F}$ endowed with its weak topology. 
\end{lemma}
\begin{proof}
Fix $u\in C_c(X)\cap \mc{F}$ and consider a sequence $\{u_n\}_{n\in \mathbb{N}}\subset \mc{F}$ so that
\begin{align}
\mc{E}(u_n)\xrightarrow{n\to +\infty} 0 \quad \quad \textnormal{and} \quad \quad u_n=1\ \textnormal{a.e. in}\ U_n,
\end{align}
where $U_n$ is a  neighbourhood of $K$. In particular we may suppose that $u_n \le 1$ a.e. in $X$ in view of the Markovianity of $\mc{E}$, see \cite[Subsection 1.1]{FM_Dirichlet_forms_book}. Then $(1-u_n)u \in Z_K$ and 
\begin{equation}
\mc{E}(u-(1-u_n)u)=\mc{E}(u_n u) \le 2 \left(\norm{u_n}_{L^\infty(X,m)}\mc{E}(u)+\norm{u}_{L^\infty(X,m)}\mc{E}(u_n)\right) 
\quad \text{ for any } n \in \mathbb{N}.
\end{equation}
Hence $\{uu_n\}_{n \in \mb{N}}$ is bounded in $\mc{F}$ in view of Proposition \ref{prop_poin}. In particular, up to a subsequence, there exists $w \in \mc{F}$  such that $uu_n \rightharpoonup w$ weakly in $\mc{F}$ as $n \to \infty$. By Proposition \ref{prop_compact_embedding} it follows that $uu_n \to w$ 
strongly in $L^2(X,m)$. Furthermore by Proposition \ref{prop_poin}
\begin{equation}
\int_{X}|uu_n|^2 \, dm\le \norm{u}_{L^\infty(X,m)}^2 \norm{u_n}_{L^2(X,m)}^2 \le \la_0^{-1}\norm{u}_{L^\infty(X,m)}^2 \mc{E}(u_n) \to 0^+, \text{ as } n \to \infty.
\end{equation}
We conclude that $w=0$ and so $(1-u_n)u \rightharpoonup u$ weakly in $\mc{F}$. Then the claim follows from \eqref{AS_density}.  
\end{proof}

In the spirit of \cite[Proposition 3.8] {FNO_disa_Dirichlet_region}, we have the following lemma which, together with Proposition \ref{prop_poin}, proves that  \eqref{hp_limit_E_e_Ve_o} holds. We recall that $ V_\e^f $ is the classical capacitary potential associated to the capacity $\capa{f}{K_\e}$.
\begin{lemma}
Under the assumptions above on $K_\e$ and $K_0$, for any $f \in \mathcal{F}$
\begin{align}
    &\capa{f}{K_\e} \to 0^+, \quad \text{ as } \e \to 0^+,\label{limit_capa}\\
    & V_\e^f \to 0 \text{ strongly in } \mc{F}, \quad \text{ as } \e \to 0^+,\label{limit_potential}\\
    & m(K_\e) \to 0^+, \quad \text{ as } \e \to 0^+. \label{limit_meausure}
\end{align}
\end{lemma}
\begin{proof}
For any $\e \in (0,1]$ the potential $V_\e^f$ solves the equation 
\begin{equation}\label{proof:lemma_capa_1}
\mc{E}(V_\e^f,w)=0  \quad \text{ for any } w \in Z_\e.
\end{equation}
Testing the equation above with $V_\e^f-f$ we conclude that, by the Cauchy-Schwarz inequality,
\begin{equation}
\mc{E}(V_\e^f)\le \mc{E}(f)  \quad \text{ for any } \e \in(0,1].
\end{equation}
Hence there exist a sequence $\{V_{\e_n}^f\}_{n \in \mb{N}}$ and $V \in \mc{F}$ such that $V_{\e_n}^f \rightharpoonup V$ weakly in $\mc{F}$.  Let $\psi \in Z_{K_0}$ and $U$ a neighbourhood of $K_0$ such that $\mathop{\rm{supp}}{\psi} \cap U=\emptyset$. Let $\e_0$ be such that $K_\e \subset U$ for any $\e \in [0,\e_0]$. Testing \eqref{proof:lemma_capa_1} with $\psi$ and passing to the limit as $n \to \infty$ we conclude that $\mc{E}(V,\psi)=0$ and so
\begin{equation}
\mc{E}(V,w)=0  \quad \text{ for any } w \in Z_{K_0}.
\end{equation}
Hence by Lemma \ref{time_consuming_lemma} it follows that $V=V_{K_0}^f=0$. Moreover,
\begin{align}
0=\capa{f}{K_0}=\mc{E}(V)=\mc{E}(V,f)=\lim_{n \to \infty }\mc{E}(V_{\e_n},f)=\lim_{n \to \infty }\capa{f}{K_{\e_n}}.
\end{align}
By the Urysohn Subsequence Principle we conclude that \eqref{limit_capa} and \eqref{limit_potential} hold.

We recall that for any set $K \subset X$
\begin{equation}
\capa{}{K}=\inf\{\capa{}{U}\ :\ U\subseteq X \text{ open},\ K \subset U\}.
\end{equation}
Fix $\delta>0$ small enough and $U\subseteq X$ open so that $K_0 \subset U$ and $\capa{}{U}\leq\delta$. For $\e_0\in(0,1]$ small enough, $K_\e\subset U$ for every $\e\in (0,\e_0]$ and so, thanks to \cite[Subsection 2.1]{FM_Dirichlet_forms_book} and Proposition \ref{prop_poin},
\begin{align}
    m(K_\e)\leq m(U)\leq (1+\la_1^{-1})\capa{}{U}\leq (1+\la_1^{-1})\delta.
\end{align}
Passing to the limit as $\e \to 0^+$
we conclude that 
\begin{equation}
\limsup_{\e \to 0^+}m(K_\e)\leq(1+\la_1^{-1})\delta,
\end{equation}
for any $\delta>0$ and so  \eqref{limit_meausure} holds.
\end{proof}

\begin{proposition}
For any $n \in \mathbb{N}$
\begin{equation}\label{limit_eigen_dirichlet}
\lim_{\e \to 0^+}\la_{\e,n}=\la_{n}.
\end{equation}
\end{proposition}

\begin{proof}
We are going to prove that the assumptions of Theorem \ref{theo_spectral_stability} are satisfied.
Let us define the linear operator $P_\e: L^2(X,m) \to  \mc{Z}_\e$   as 
\begin{equation}
 P_\e(u)(x):=
 \begin{cases}
     u(x), & \text{ if } x \in X \setminus K_\e,\\
     0, & \text{ if } x \in  K_\e.
 \end{cases}
\end{equation}
Then for any $u \in L^2(X,m)$, thanks to Proposition \ref{prop_poin},
\begin{align}
    \mc{E}(R_\e P_\e u)&=(P_\e u, R_\e P_\e u)_{L^2(X,m)}\\
    &\leq \norm{u}_{L^2(X,m)}\norm{R_\e P_\e u}_{L^2(X,m)}\\
    &\leq \la_1^{-\frac{1}{2}}  \norm{u}_{L^2(X,m)}\mc{E}(R_\e P_\e u)^{\frac{1}{2}}.
\end{align}
Furthermore if $R_{\e_n} P_{\e_n} u \rightharpoonup w$  weakly in $\mc{F}$ for some $w \in \mc{F}$ as $n\to \infty$, then clearly
\begin{equation}
\lim_{n \to \infty} \mc{E}(R_{\e_n} P_{\e_n} u,v)= \mc{E}(w,v) \quad \forall v \in Z_0.
\end{equation}
We conclude that \eqref{limit_eigen_dirichlet} holds in view of Theorem \ref{theo_spectral_stability}
\end{proof}

We have proved that all the hypotheses of Theorem \ref{theo_exp_1_order} are satisfied. Furthermore by Lemma \ref{time_consuming_lemma} and Proposition \ref{prop_norm_L2_o} 
\begin{equation}
\norm{V_\e}^2_{L^2(X,m)}=o(\mc{E}(V_\e))=o(\capa{\phi}{K_\e}), \quad \text{ as }\e \to 0^+,
\end{equation}
and by \eqref{eq_Ve_Dir_form} tested with $V_\e-\phi$ and \eqref{eq_phi0_Dir_form} tested with $V_\e$
\begin{equation}
\capa{\phi}{K_\e}=\mc{E}(V_\e,V_\e)=\la(V_\e,\phi).
\end{equation}
In conclusion, by Theorem \ref{theo_exp_1_order}, we have shown that 
\begin{equation}
\la_\e-\la=\capa{\phi}{K_\e}+o(\capa{\phi}{K_\e}), \quad \text{ as } \e \to 0^+,
\end{equation}
where we are denoting with $\la_{\e}$ the simple eigenvalue $\la_{\e, n}$ for any $ \e \in (0,1]$.

\subsection{Pseudo-differential operators}\label{subsec_pseudo}
Let $\Omega \subset \mathbb{R}^N$ be a bounded Lipschitz domain. Consider a family $\{H_\e\}_\e$ of pseudo-differential operators with domains
\begin{align}
&L^2_0(\Omega):=\{u \in L^2(\mathbb{R}^N): u=0\ \inn \R^N\setminus \Omega\}\\
&\textnormal{Dom}(H_\e):=\{u \in L^2_0(\Omega)\ :\ f_\e \widehat{u} \in L^2(\R^N)\},
\end{align}
where $\widehat{u}$ is the Fourier transform of $u$ and $f_\e$ denotes the symbol of $H_\e$. We will also use $\textgoth{F}$ to denote the Fourier transform. 
On the symbol $f_\e$ we assume that 
\begin{equation}\label{hp_fe_intrìegrability}
f_\e \in L^2_{loc}(\R^N) \quad \text{ and } \quad f_\e (\xi) =O(|\xi|^m), \text{ for some }  m \in \mathbb{N} \text{ as } |\xi| \to \infty.  
\end{equation}
Let $\mathcal{E}^{(\e)}$ be the bilinear form associated to $H_\e$
\begin{align}
    \mc{E}^{(\e)}(u,v):=\int_{\R^N} f_\e (\xi) \widehat{u}(\xi) \overline{\widehat{v}(\xi)} \, \textnormal{d}\xi
\end{align}
with domain
\begin{align}
    \mc{F}_\e:=\{u \in L^2_0(\Omega)\ :\ f_\e^{1/2} \widehat{u} \in L^2(\R^N)\}.
\end{align}
In what follows we assume that for every $\e\in[0,1]$ the symbol $f_\e$ of the operator $H_\e$ satisfies the following properties:
\begin{enumerate}
    \item $f_\e>1$ almost everywhere in $\R^N$;
    \item there exists a positive constants $C_1,C_0>0$ such that 
    \begin{equation}\label{hp_estimates_fe}
        C_1  f_1(\xi) \le    f_\e(\xi)\leq C_0 f_0(\xi) \quad \text{ almost everywhere in }  \R^N.
    \end{equation}
\end{enumerate}
Moreover, we require that
\begin{align}\label{fe_to_f0_pseudo}
    f_\e \xrightarrow[]{\e \to 0} f_0 \quad \textnormal{almost everywhere in}\ \R^N.
\end{align}
The first assumption is not restrictive, as observed after assumption \eqref{hp_la_1>0}, and will assure  that the bottom of the spectrum is bigger or equal than $1$. In particular the resolvent operator $R_\e=H_\e^{-1}$ is well defined.  We also assume that $R_\e$ is compact. A simple criterion to assure the validity of this assumption is given in the next proposition.

\begin{proposition}\label{prop_comp_embedding_pseudo}
Suppose that 
\begin{equation}\label{hp_fe_compc}
\int_{\R^N} e^{-t f_\e^{\frac{1}{2}}} \, \textnormal{d} \xi  <+\infty \quad \text{ for any } t>0.
\end{equation}
Then the embedding $\mc{F}_\e \hookrightarrow L_0^2(\Omega)$ is compact. In particular, $R_\e$ is a compact operator.
\end{proposition}
\begin{proof}
Let $ u \in \mathcal{F}_\e$ and $t>0$. We define 
\begin{equation}
e^{-t \sqrt{H}_\e} u:= \textgoth{F}^{-1}\left(  e^{-tf_\e^{\frac{1}{2}}(\xi)} \widehat{u}(\xi)\right).
\end{equation}
We divide the proof in two steps.

\textbf{Step 1.} We claim that for any $t>0$ and $u \in \mc{F}_\e$
\begin{equation}
\norm{u-e^{-t \sqrt{H}_\e} u}_{L^2(\R^N)} \le t (\mc{E}^{(\e)}(u))^{\frac{1}{2}}.
\end{equation}
Indeed by the Plancherel identity 
\begin{align}
\int_{\R^N} |u-e^{-t \sqrt{H}_\e} u|^2 \, \textnormal{d}x &=\int_{\R^N} |1-e^{-t (f_\e(\xi))^{\frac{1}{2}}} |^2 |\widehat{u}(\xi)|^2 \, \textnormal{d}\xi\\
&\le t^2 \int_{\R^N} |f_\e(\xi)| |\widehat{u}(\xi)|^2 \, \textnormal{d}\xi.
\end{align}

\textbf{Step 2.}  Let $\{u_n\}_{n \in \mb{N}}$ and $u$  be such that $u_n \rightharpoonup u$ weakly in $\mathcal{F}_\e$ as $n \to \infty$. We claim that $u_n \to u$ strongly $L^2_0(\Omega)$ as $n \to \infty$. To this end we notice that 
\begin{equation}
\norm{u-u_n}_{L^2(\Omega)}  \le \norm{u-e^{-t \sqrt{H}_\e} u}_{L^2(\Omega)}+\norm{u_n-e^{-t \sqrt{H}_\e} u_n}_{L^2(\Omega)}+
\norm{e^{-t \sqrt{H}_\e} (u-u_n)}_{L^2(\Omega)}
\end{equation}
and so by Step 1 and the fact that $\{u_n\}$ is bounded in $\mc{F}_\e$, it is enough to prove that 
\begin{equation}\label{proof:compct_pseudo_0}
 \lim_{n \to \infty}\norm{e^{-t \sqrt{H}_\e} (u-u_n)}_{L^2(\Omega)} =0.
\end{equation}
We notice that for any $v \in \mathcal{F}_\e$  
\begin{equation}\label{proof:compct_pseudo}
e^{-t \sqrt{H}_\e} v= \textgoth{F}^{-1}\left(  e^{-tf_\e^{\frac{1}{2}}(\xi)}\right)* v
\end{equation}
and so  by the Young inequality and \eqref{hp_fe_compc}
\begin{equation}
\norm{e^{-t \sqrt{H}_\e} v}_{L^\infty(\R^N)} \le \norm{v}_{L^2(\R^N)} \norm{e^{-t f_\e^{\frac{1}{2}}}}_{L^2(\R^N)}.
\end{equation}
 Furthermore, since $u_n \rightharpoonup u$ weakly in $L^2_0(\Omega)$ as $n \to \infty$,  by \eqref{proof:compct_pseudo}, it follows that   
\begin{equation}
  e^{-t \sqrt{H}_\e}u_n \to e^{-t \sqrt{H}_\e}u \quad \text{ a.e. in }\R^N.  
\end{equation}
Then by Dominated Converge Theorem  we conclude that \eqref{proof:compct_pseudo_0} holds.
Since $R_\e: L^2_0(\Omega) \to \mc{F}_\e$ is continuous then is clear that $R_\e: L^2_0(\Omega) \to  L^2_0(\Omega) $ is compact.
\end{proof}

As usual we denote with $\{\la_{\e,n}\}_{n \in \mb{N}}$ the spectrum of $H_\e$ for $\e \in [0,1]$.
Similarly we denote with $\{\phi_{0,n}\}_{n \in \mb{N}}$ an orthonormal basis of eigenfunctions of $H_0$ of  $L^2_0(\Omega)$.
\begin{proposition}
For any $n \in \mathbb{N}$
\begin{equation}\label{lim_eigen_pseudo} 
\lim_{\e \to 0^+} \la_{\e,n}=\la_{0,n}.
\end{equation}
\end{proposition}
\begin{proof}
Thanks to \eqref{hp_estimates_fe}, for any $u \in L^2_0(\Omega)$
\begin{equation}
\mc{E}^{(1)}(R_\e u) \le \frac{1}{C_1} \mc{E}^{(\e)}(R_\e u) =\frac{1}{C_1}  (u,R_\e u)_{L^2(\Omega)} \le \frac{1}{C_1} \norm{u}_{L^2(\Omega)}.
\end{equation}
Furthermore if $R_{\e_n} u \rightharpoonup w$  weakly in $\mc{F}_1$ for some $w \in \mc{F}_1$ as $n\to \infty$, then, up to a subsequence, 
\begin{equation}
\int_{\R^N} f_0 |\widehat{w}|^2\ \textnormal{d}\xi \le \liminf_{n \to \infty}  \int_{\R^N} f_{\e_n} |\widehat{R_{\e_n} u}|^2\, \textnormal{d}\xi
\le \frac{1}{C_1} \norm{u}_{L^2(\Omega)},
\end{equation}
by Fatou's Lemma and Proposition \ref{prop_comp_embedding_pseudo}. Since, up to a subsequence, there exists a function $g \in L^2(\R^N)$ such that $|\widehat{R_{\e_n} u}| \le g$ a.e. in $\R^N$, by the Dominated Convergence Theorem 
\begin{equation}
\lim_{n \to \infty} \int_{\R^N} f_{\e_n} \widehat{R_{\e_n} u}\widehat{v}\ \textnormal{d}\xi =  
\int_{\R^N} f_0 \widehat{w}\widehat{v} \ \textnormal{d}\xi.
\end{equation}
By Theorem \ref{theo_spectral_stability}, we conclude that \eqref{lim_eigen_pseudo} holds.
\end{proof}

Fix a simple eigenvalue $\la_0:=\la_{n_0,0}$ and a correspondent eigenfunction $\phi_0:=\phi_{n_0,0}$.
\begin{proposition}
For any $u \in \mc{F_\e}$ we have that 
\begin{equation}\label{eq_Le_pseudo}
\mc{E}^{(\e)}(\phi_0, u)= \la_0 \int_{\Omega} \widehat{\phi}_0 \widehat{u} \ \textnormal{d}\xi + L_\e(u),
\end{equation}
where
\begin{align*}
    L_\e(u)=\int_{\R^N} (f_\e -f_0) \widehat{\phi}_0 \widehat{u} \ \textnormal{d}\xi.
\end{align*}
\end{proposition}
\begin{proof}
Let $u \in \mc{F_\e}$. Then clearly
\begin{equation}
\mc{E}^{(\e)}(\phi_0, u)= \int_{\R^N} f_0 \widehat{\phi}_0 \widehat{u} \ \textnormal{d}\xi + \int_{\R^N} (f_\e -f_0) \widehat{\phi}_0 \widehat{u} \ \textnormal{d}\xi
\end{equation}
and the previous equation makes sense  in view of Proposition \ref{prop_regularity}. Let $\{u_n\}_{n\in \mathbb{N}} \subset C^\infty_c(\Omega)$ be a sequence of functions such that $u_n \to u$ strongly in $L^2(\Omega)$. Hence, by Proposition \ref{prop_regularity},
\begin{equation}
\int_{\R^N} f_0 \widehat{\phi}_0 \widehat{u} \ \textnormal{d}\xi= \lim_{n\to \infty} \int_{\R^N} f_0 \widehat{\phi}_0 \widehat{u}_n \ \textnormal{d}\xi= 
\lim_{n\to \infty}\la_0\int_{\R^N}\phi_0 u_n \ \textnormal{d}x=\la_0\int_{\R^N}\phi_0 u\ \textnormal{d}\xi
\end{equation}
which proves \eqref{eq_Le_pseudo}.
\end{proof}

Let us define 
\begin{equation}
J_\e(u)= \frac{1}{2} \mc{E}^{(\e)}(u)-L_\e(u).
\end{equation}
Then, in view of Proposition \ref{prop_J_min}, there exists a function $V_\e \in \mc{F}_\e$ that minimizes $J_\e$ over $\mc{F}_\e$ .
\begin{proposition}
We have that 
\begin{equation}\label{eq_norm_Ve_to0_pseduo}
\lim_{\e \to 0^+}\mc{E}^{(\e)}(V_\e)=0.
\end{equation}
\end{proposition}
\begin{proof}
Testing \eqref{eq_Ve} with $V_\e$ we obtain
\begin{equation}\label{Eq_estimate_vanishing_pseudo}
\mc{E}^{(\e)}(V_\e) \le \left(\int_{\R^N}|f_\e-f_0|^2  |\widehat{\phi}_0|^2 \ \textnormal{d} \xi\right)^{\frac{1}{2}}
\left(\int_{\R^N}f_\e |V_\e|^2 \ \textnormal{d} \xi\right)^{\frac{1}{2}},
\end{equation}
in view of the H\"older inequality and  the fact that $f_\e>1$. By  Proposition \ref{prop_regularity}, the Dominated Convergence Theorem, and \eqref{fe_to_f0_pseudo} we conclude that \eqref{eq_norm_Ve_to0_pseduo} holds.
\end{proof}

In conclusion, we are in position to apply Theorem \ref{theo_exp_1_order} thus obtaining 
\begin{equation}\label{eq_eigen_pseudo}
\la_\e -\la_0= \la_0 \int_{\Omega} \phi_0 V_\e \ \textnormal{d}x + O(\norm{V_\e}_{L^2(\Omega)}^2) \quad  \text{ as } \e \to 0^+, 
\end{equation}
in view of Remark \ref{Rmk:Accurate1Exp}.

If the rate of convergence of $f_\e \to f_0$ can be quantified, we can compute the vanishing order of $\la_0 \int_{\Omega} \phi_0 V_\e \, dx$. More precisely, we assume that there exists $\e_0 \in (0,1]$ such that 
\begin{align}
&\lim_{\e \to 0^+}\frac{f_\e(\xi)-f_0(\xi)}{\e}=h(\xi),  \label{hp_fe_fo_precise} \\
&\left|\frac{f_\e(\xi)-f_0(\xi)}{\e}\right|\le  C (1+|h(\xi)|) \quad \text{ for a.e. } \xi \in \R^N, \text{ for any } \e \in (0,\e_0), \label{hp_fe_fo_incremental_ratio}
\end{align}
and
\begin{align}\label{hp_h_fo}
|h(\xi)|^2\le C f_0(\xi)^2 f_\e(\xi) \quad \text{ for a.e. } \xi \in \R^N, \text{ for any } \e \in [0,\e_0)
\end{align}
for some measurable functions $h: \R^N \to \R$ and constant $C>0$. By the Lagrange Theorem,  the assumption $\left|\frac{f_\e(\xi)-f_0(\xi)}{\e}\right|\le C (1+|h(\xi)|)$ is verified, when, for example,  $\e \to f_\e(\xi)$ is derivable for any $\e \in (0,\e_0)$, and  
$\left|\pd{f_\e}{\e}(\xi)\right|\le C (1+|h(\xi)|)$ for some constant $C>0$ that does not depend on $\e$.

\begin{proposition}\label{prop_convergence_Ve_pseudo}
There exists $V \in \mc{F}_0$ such that $\frac{V_\e}{\e} \rightharpoonup V$ weakly in $\mc{F}_1$ and $V$ solves the equation 
\begin{equation}
\int_{\R^N} f_0 \widehat{V}\widehat{u}\ \textnormal{d} \xi =\int_{\R^N} h \widehat{\phi}_0 \widehat{u}\ \textnormal{d} \xi, \quad \text{ for any } u \in {\rm{Dom}}[H_0].
\end{equation}
\end{proposition}
\begin{proof}
Thanks to \eqref{hp_estimates_fe}, \eqref{hp_fe_fo_precise}, Proposition \ref{prop_regularity}, and \eqref{Eq_estimate_vanishing_pseudo} the family
$\left\{\frac{V_\e}{\e}\right\}_{\e \in [0,1]}$ is bounded in $\mc{F}_1$. In particular there exists a sequence $\e_n \to 0^+$ and $V \in \mc{F}_1$ such that $\frac{V_\e}{\e} \rightharpoonup V$ weakly in $\mc{F}_1$ as $n \to \infty$ and so,  by Proposition \ref{prop_compact_embedding},  $\frac{V_\e}{\e} \to V$ strongly in $L^2(\R^N)$. 

Equation \eqref{eq_Ve} in this case is
\begin{equation}\label{Eq_Ve_pseudo}
\int_{\R^N} f_\e \widehat{V}_\e \widehat{u} \ \textnormal{d} \xi = \int_{\R^N} (f_\e-f_0)  \widehat{\phi}_0 \widehat{u} \ \textnormal{d} \xi \quad \text{ for any } u  \in \mc{F}_\e.
\end{equation}
Furthermore, up to pass to a further subsequence, by Fatou's Lemma, 
\begin{align}
\int_{\R^N} f_0 \widehat{V}^2 \, d\xi\le\liminf_{n\to \infty} \int_{\R^N}  f_\e \frac{\widehat{V}_{\e_n}^2}{{\e_n}^2} \ \textnormal{d} \xi 
\le\liminf_{n\to \infty} \int_{\R^N}  \frac{|f_{\e_n}-f_0|}{\e_n}\frac{\widehat{V}_{\e_n}}{{\e_n}} \widehat{\phi}_0 \ \textnormal{d} \xi 
\end{align}
and hence, by the Cauchy-Schwarz inequality, 
\begin{align*}
    \liminf_{\e_n \to 0^+} \int_{\R^N} f_{\e_n} \frac{\widehat{V}_{\e_n}^2}{{\e_n}^2}\ \textnormal{d}\xi \leq \lim_{n \to \infty} 
    \left(\int_{\R^N} \frac{(f_0-f_{\e_n})^2}{ \e_n^2} f^{-1}_{\e_n}\widehat{\phi}_0^2 \ \textnormal{d}\xi\right)^{\frac{1}{2}}.
\end{align*}
We conclude that $V \in \mathcal{F}_0$ in view of \eqref{hp_fe_fo_precise}, \eqref{hp_fe_fo_incremental_ratio}, \eqref{hp_h_fo} and Proposition \ref{prop_regularity}.
Multiplying by $\e_n^{-1}$ and passing to the limit as $n \to \infty$ in \eqref{Eq_Ve_pseudo}
we obtain 
\begin{equation}
\int_{\R^N} f_0 \widehat{V} \widehat{u} \ \textnormal{d} \xi = \int_{\R^N} h\widehat{\phi}_0 \widehat{u} \ \textnormal{d} \xi \quad \text{ for any } u  \in \mc{F}_0,
\end{equation}
thanks to \eqref{hp_fe_fo_precise}, \eqref{hp_fe_fo_incremental_ratio}, \eqref{hp_h_fo} and the Dominated Convergence Theorem. Since the solution of the above equation is unique, we conclude that $\frac{V_\e}{\e} \rightharpoonup V$ weakly in $\mc{F}_1$ by the Urysohn Subsequence Principle.
\end{proof}

In conclusion,  by the Plancherel identity and  Proposition \ref{prop_convergence_Ve_pseudo},
\begin{equation}
\frac{\la_0}{\e} \int_{\Omega} \phi_0 V_\e \ \textnormal{d}x = \la_0\int_{\R^N} \widehat{\phi}_0 \frac{\widehat{V}_\e}{\e} \ \textnormal{d}\xi \to 
 \la_0\int_{\R^N} \widehat{\phi}_0 \widehat{V} \ \textnormal{d}\xi = 
\int_{\R^N}f_0 \widehat{\phi}_0 \widehat{V} \ \textnormal{d}\xi=\int_{\R^N}h|\widehat{\phi}_0|^2\ \textnormal{d}\xi.
\end{equation}
Hence 
\begin{align}
    \la_\e -\la_0= \e \int_{\R^N}h|\widehat{\phi}_0|^2\ \textnormal{d}\xi+ O(\e^2) \quad  \text{ as } \e \to 0^+. 
\end{align}

\begin{example}
Let $\phi_{n,0}$ be the eigenfunction associated to the simple  eigenvalue $\la_{n,0}$ of the Dirichlet Laplacian on a bounded domain $\Omega$.
Choosing the symbol $f_{\e}(\xi):=1+|\xi|^{2-2\e}$ for any $\e \in [0,1]$, we have that $h(\xi)=-2\log(|\xi|)|\xi|^2$, where $h$ is  as in \eqref{hp_fe_fo_precise}.
Furthermore 
\begin{equation}
\left|\pd{f_\e(\xi)}{\e}\right|= -2\log(|\xi|)|\xi|^{2-2\e} \le C(1+2|\log(|\xi|)| \, | \xi|^{2}), \quad \text{ for any }\xi \in \R^N 
\end{equation}
for some positive constant $C>0$. We conclude that \eqref{hp_fe_fo_precise} holds.

Hence we have the following asymptotic for any simple eigenvalue $\la_{n,\e }$ of the fractional Laplacian $(-\Delta)^{1-\e}$ on a bounded domain with Dirichlet boundary conditions:
\begin{align}
   \la_{n,\e} -\la_{n,0}=-2 \e \int_{\R^N}\log(|\xi|)|\xi|^2 |\widehat{\phi}_ {n,0}(\xi)|^2\ \textnormal{d}\xi+ O(\e^2) \quad  \text{ as } \e \to 0^+.
\end{align}
In conclusion, we have  quantified  the spectral stability  result obtained in \cite{BPS_fractional_stability} in the special case of the fractional Laplacian.
\end{example}

\bigskip\noindent {\bf Acknowledgments.}  
The authors are members of GNAMPA-INdAM.
G. Siclari is partially supported by the MUR-PRIN project no. 20227HX33Z 
``Pattern formation in nonlinear phenomena'' granted by the European Union -- Next Generation EU.
The authors would also like to thank Veronica Felli and  Luigi Provenzano for their helpful suggestions which helped improve the manuscript.

\bibliographystyle{acm}
\bibliography{Dirichlet_forms}	
\end{document}